\documentclass{article}

\usepackage{xcolor}
\usepackage{amsmath}
\usepackage{amsfonts}
\usepackage{amssymb}
\usepackage{amsthm}
\usepackage{array}
\usepackage{stmaryrd}
\usepackage{tikz-cd, tikz}
\usetikzlibrary{calc}
\usepackage{graphicx}
\usepackage{thmtools}
\usepackage{thm-restate}
\allowdisplaybreaks
\usepackage{hyperref}
\hypersetup{
}
\usepackage{cleveref}
\usepackage[disable]{todonotes}
\newtheorem{theorem}{Theorem}
\newtheorem*{theorem*}{Theorem}
\newtheorem{lemma}[theorem]{Lemma}
\newtheorem*{lemma*}{Lemma}
\newtheorem{proposition}[theorem]{Proposition}
\newtheorem*{proposition*}{Proposition}
\newtheorem{corollary}[theorem]{Corollary}
\theoremstyle{definition}
\newtheorem{definition}[theorem]{Definition}
\theoremstyle{remark}
\newtheorem{remark}[theorem]{Remark}

\newcommand{\Z}{\mathbf{Z}}

\newcommand{\Q}{\mathbf{Q}}
\newcommand{\Qbar}{\overline{\Q}}
\newcommand{\R}{\mathbf{R}}
\newcommand{\C}{\mathbf{C}}
\newcommand{\N}{\mathbf{N}}
\newcommand{\F}{\mathbf{F}}
\newcommand{\A}{\mathbf{A}}

\newcommand{\calO}{\mathcal{O}}
\newcommand{\fraka}{\mathfrak{a}}
\newcommand{\frakl}{\mathfrak{l}}
\newcommand{\frakm}{\mathfrak{m}}

\newcommand{\frakp}{\mathfrak{p}}
\newcommand{\frakP}{\mathfrak{P}}

\newcommand{\End}{\textnormal{End}}

\newcommand{\Frob}{\textnormal{Frob}}
\newcommand{\Gal}{\textnormal{Gal}}

\newcommand{\Norm}{\textnormal{Norm}}

\newcommand{\ord}{\textnormal{ord}}

\newcommand{\red}{\textnormal{red}}

\newcommand{\Sym}{\textnormal{Sym}}
\newcommand{\cyclo}{\omega}

\newcommand{\rank}{\textnormal{rank}}

\newcommand{\Trace}{\textnormal{Trace}}

\newcommand{\GL}{\textnormal{GL}}
\newcommand{\SL}{\textnormal{SL}}
\newcommand{\PGL}{\textnormal{PGL}}
\newcommand{\Frac}{\textnormal{Frac}}
\newcommand{\Spec}{\textnormal{Spec}}

\newcommand{\Realpart}{\operatorname{Re}}
\newcommand{\house}[1]{%
  \tikz[baseline]{\node[anchor=base,inner sep=0.3ex](mynode){\ensuremath{\vphantom{A}#1}};
  \draw(mynode.south west)--(mynode.north west)--(mynode.north east)--(mynode.south east);
  \path[use as bounding box]($(mynode.south west)+(-0.3ex,-0.3ex)$)rectangle($(mynode.north east)+(0.3ex,0.3ex)$);}
} 

\usepackage[backend=bibtex]{biblatex}
\title{Classical forms of weight one in ordinary families}
\author{Eric Stubley}
\date{\today}

\begin{document}
\maketitle
\begin{abstract}
We develop a new strategy for studying low weight specializations of $p$-adic families of ordinary modular forms.
In the elliptic case, we give a new proof of a result of Ghate--Vatsal which states that a Hida family contains infinitely many classical eigenforms of weight one if and only if it has complex multiplication.
Our strategy is designed to explicitly avoid use of the related facts that the Galois representation attached to a classical weight one eigenform has finite image, and that classical weight one eigenforms satisfy the Ramanujan conjecture.
We indicate how this strategy might be used to prove similar statement in the case of partial weight one Hilbert modular forms, given a suitable development of Hida theory in that setting.
\end{abstract}

\tableofcontents

\section{Introduction}\label{sec:introduction}

% fundamental question to start with: specialization in weight k >= 2 always classical; what about weight one?
The goal of this article is to provide a new proof of a theorem of Ghate and Vatsal, which states that a $p$-adic family of $p$-ordinary modular forms contains infinitely many classical forms of weight one if and only if the family has complex multiplication.
Our reason for seeking out a new proof of this result is to avoid the use of the related facts that classical eigenforms of weight one satisfy the Ramanujan conjecture and their associated Galois representations have finite image.
We wish to avoid using these facts so as to provide a proof strategy which is likely to generalize to the case of partial weight one Hilbert modular forms, where the Ramanujan conjecture is as yet unknown.

Let us formulate this result precisely.
We fix an odd prime number $p$, and let $\Lambda = \Z_{p}\llbracket 1 + p\Z_{p}\rrbracket$ denote the Iwasawa algebra, which we think of as parametrizing continuous $p$-adic characters of the multiplicative group $1 + p\Z_{p}$.

We say that a classical modular eigenform is said to be $p$-ordinary if its $U_{p}$-eigenvalue is a $p$-adic unit.
We will be interested in $p$-adic families of $p$-ordinary modular eigenforms.
Pick a natural number $N$ coprime to $p$, which will serve as the prime-to-$p$ level of the modular forms we study.
Hida's theory of $p$-ordinary families produces a finite free $\Lambda$-module $\mathbf{H}^{\ord}(N; \Z_{p})$ containing elements $T_{\ell}$ for every prime $\ell \nmid Np$ and $U_{\ell}$ for each prime $\ell | Np$, which is universal in the following sense.
For any normalized modular eigenform $f$ which is of weight $k \geq 2$, level of the form $Np^{r}$ for some $r \geq 0$, and which is $p$-ordinary, there is a unique homomorphism
\[
\mathbf{H}^{\ord}(N; \Z_{p}) \to \Qbar_{p}
\]
determined by sending the elements $T_{\ell}$ (or $U_{\ell}$ if $\ell|Np$) to the $\ell$-th Hecke eigenvalue of $f$.
For such $f$, we say that $\mathbf{H}^{\ord}(N; \Z_{p})$ specializes to $f$.

The ordinary Hecke algebra $\mathbf{H}^{\ord}(N; \Z_{p})$ is constructed as a Hecke algebra acting on a large space of $p$-adic modular forms.
It is known that any eigensystem $\mathbf{H}^{\ord}(N; \Z_{p}) \to \Qbar_{p}$ which is in arithmetic weight $k \geq 2$ (meaning the composite map $\Lambda \to \mathbf{H}^{\ord}(N; \Z_{p}) \to \Qbar_{p}$ is a finite order character times the $(k-1)$-st power of the cyclotomic character) is the eigensystem of a classical, rather than just a $p$-adic, modular form.
However this is no longer true for eigensystems in weight one (where the corresponding map $\Lambda \to \Qbar_{p}$ is a finite order character): such an eigensystem may or may not be that of a classical modular form.

The following theorem of Ghate and Vatsal characterizes exactly when a component of $\mathbf{H}^{\ord}(N; \Z_{p})$ admits infinitely many classical weight one specializations.

% what is the best way to communicate the necessary definitions: p is an odd prime, |lambda is the group ring, etc.
% on one hand they are all familiar, on the other you should be precise
\begin{theorem}[Ghate--Vatsal, cf. proposition 14 of \cite{ghate_vatsal}]\label{thm:intro_elliptic}
Let $\mathbf{I}$ be a reduced irreducible component of the ordinary $\Lambda$-adic Hecke algebra $\mathbf{H}^{\ord}(N; \Z_{p})$.
Then $\mathbf{I}$ specializes to infinitely many classical eigenforms of weight one if and only if $\mathbf{I}$ has complex multiplication.
\end{theorem}

In this article we give a new proof of this result.
The method of Ghate--Vatsal relies crucially on the fact that the Galois representations associated to classical eigenforms of weight one have finite image.
As we are interested in generalizing this result to the situation of Hilbert modular forms of partial weight one, where the associated Galois representations have infinite image and the Ramanujan conjecture is still open, our new proof of this result in the elliptic case avoids both of these known facts for classical eigenforms of weight one.

\subsection{Overview of the strategy}\label{sec:strategy}

Our strategy is to utilize a result of Hida from \cite{hida_hilbert_i} which characterizes whether or not a $p$-ordinary family has complex multiplication (hereafter abbreviated as ``CM'') by the arithmetic complexity of the Hecke fields of the classical forms it interpolates.
A version of this result is restated in this article as \cref{thm:hida_cm_iff_hecke_field} and its proof is sketched in \cref{sec:bounded}.

In order to apply this characterization of CM families, we rely on a crucial fact: Hecke fields of $p$-ordinary classical weight one eigenforms cannot be too complicated.
This is shown in \cref{lem:weight_one_hecke_field_bound} where we prove that for a $p$-ordinary classical weight one eigenform $f \in S_{1}(N, \epsilon; \Q_{p}(\epsilon))$ we have that
\[
[\Q(f):\Q(\epsilon)] \leq \rank_{\Lambda}(\mathbf{H}^{\ord}(N; \Z_{p})) < \infty.
\]
This is a consequence of the fact that any classical weight one eigenform either has $U_{p}$ eigenvalue equal to $0$ or is $p$-ordinary.

If we were willing to apply Hida's characterization of CM families directly in weight one this bound on Hecke fields would suffice to prove that a $p$-ordinary family with infinitely many classical weight one specializations has CM.
However, Hida's characterization crucially uses that the Frobenius eigenvalues in the Galois representations attached to elliptic modular forms are Weil numbers, which is a consequence of the Ramanujan conjecture.
Since we seek a proof which avoids using the Ramanujan conjecture in weight one, we cannot directly apply Hida's characterization in weight one.

Instead, we propagate information about the boundedness of Hecke fields along the family from weight one into higher weights, where we do allow ourselves to make use of the Ramanujan conjecture and hence apply Hida's characterization.
The main idea is to analyze the structure of the ``algebraic power series'' (elements of finite extensions of $\Lambda$) which define Frobenius eigenvalues across the family.
The prototypical example of the type of rigidity result we employ is the following, which is also used by Hida in establishing the characterization of CM families: a power series $F(T) \in \Lambda$ for which $F(\zeta - 1)$ is a power of $\zeta$ for infinitely many $\zeta \in \mu_{p^{\infty}}$ must be of the form $\zeta'(1 + T)^{e} = \zeta' \sum_{n=0}^{\infty} \binom{n}{e} T^{n}$, where $\zeta' \in \mu_{p^{\infty}}$ and $e \in \Z_{p}$.
Our situation requires controlling algebraic power series with many specializations which are a sum of a bounded number of roots of unity rather than just a single root of unity, to account for the fact that Hecke fields of our $p$-ordinary weight one forms are uniformly bounded over $\Q(\mu_{p^{\infty}})$, but may not necessarily be cyclotomic themselves.
Once these algebraic power series have been controlled using the boundedness of Hecke fields in weight one, we use our exact knowledge of the shape of these power series to establish bounds on Hecke fields in higher weight, from whence we may apply Hida's characterization.

\subsection{Outline}\label{sec:outline}

In \cref{sec:modular_forms} we recall the key facts about modular forms which are used in our arguments.
We discuss Galois conjugation of modular forms and its relation to Hecke fields.
Our arguments rely on two types of bounds on the Hecke eigenvalues of modular forms: Archimedean bounds (as embodied by bounds coming from the holomorphicity of the $L$-series of $f$) and non-Archimedean bounds (as embodied by the slope bounds $0 \leq \ord_{p}(a_{p}(f)) \leq k-1$ for classical forms of weight $k$).

\Cref{sec:hida} covers the statements we need from the theory of ordinary families of $p$-adic modular forms.
We cover what we need in the elliptic case, paying particular attention to the inclusion of forms of weight one into ordinary families, as the literature often works only with forms of weight $k \geq 2$.

\Cref{sec:rigidity} is where new results start appearing.
This chapter focuses on rigidity principles for $p$-adic power series and integral extensions of power series rings.
After recalling in \cref{sec:weierstrass_and_newton} some facts about Weierstrass preparation and Newton polygons to set the stage for how we approach thinking about elements of integral extension of power series rings, we prove our main rigidity result in \cref{sec:bounded_sum_rigidity}.
% This main result is that if an element of an integral extension of $\Lambda$ specializes to a bounded number of roots of unity at infinitely many inputs of the form $\zeta - 1$ then it must be a linear combination of ``exponential'' power series $(1 + T)^{e} = \sum_{n=0}^{\infty} \binom{e}{n} T^{n}$. 
We conclude this chapter with tools for studying the fields of definition of the values of algebraic power series of the form determined by our rigidity results.
% We conclude this chapter with some tools for working with linear combinations of exponential power series, in particular for studying the fields of definition of their values.

\Cref{sec:construction} works throughout with a component $\mathbf{I}$ of the ordinary $\Lambda$-adic Hecke algebra, which is assumed to admit infinitely many classical weight one specializations.
This chapter covers the construction of a high-dimensional Galois representation whose characteristic polynomials of Frobenius provide a way to propagate information about Hecke fields from low weight into regular weight.
\Cref{sec:selecting_components} carries out the actual construction, which is a careful selection of components of the Hecke algebra by an extended pigeonhole principle argument; the Galois representation we want is the direct sum of the representations attached to a well-chosen set of such components.
In \cref{sec:big_rep_rigidity} we show that the characteristic polynomials of Frobenius of our high-dimensional Galois representation satisfy the conditions necessary to apply the rigidity principles of \cref{sec:rigidity}.

\Cref{sec:bounded} sketches a proof of Hida's characterization of CM families in \cref{sec:hida_bounded} and then assembles the ingredients from \cref{sec:construction} in order to apply Hida's theorem to families containing infinitely many classical forms of low weight, proving our main theorem.

We end in \cref{sec:hilbert} with a brief section discussing how the strategy in this article might be applied to $p$-adic families containing Hilbert modular forms of partial weight one.

\subsection{Notation}\label{sec:notation}
% 
% % isomorphisms for comparing between C and C_p
% For each prime $p$ we fix an isomorphism $\iota_{p}: \C \to \C_{p}$.
% These isomorphisms are used in two places.
% First they are used in constructing the $p$-adic Hecke character associated to an Archimedean Hecke character.
% % Second, they are used to allow us to index the weights of a Hilbert modular form by both the infinite places and the $p$-adic places of a totally real field $F$ in which the prime $p$ splits completely.
% 
% % p-adic valuation
Throughout the article we work with a fixed odd prime $p$.
We fix an algebraic closure $\Qbar_{p}$ of $\Q_{p}$, and completion $\C_{p}$ of $\Qbar_{p}$.
We fix the $p$-adic valuation $\ord_{p}$ on $\C_{p}$, normalized so that $\ord_{p}(p) = 1$.

% 
% % absolute Galois groups
For any field $F$ we let $G_{F}$ denote the absolute Galois group of $F$, $\Gal(\overline{F}/F)$.
We will only deal with absolute Galois groups of finite extensions of $\Q$ and $\Q_{\ell}$ for primes $\ell$.
% In this thesis we will need this notion only for finite extensions of $\Q$ or $\Q_{p}$, so issues of separability do not come into play.

\subsection{Acknowledgements}\label{sec:acknowledgements}

This article contains the results of the author's University of Chicago Ph.D. thesis.
First and foremost, thanks go to the author's advisor Frank Calegari for suggesting this problem, providing a great deal of mathematical and professional wisdom during the author's time in Chicago, and for allowing the author to approach the problem at his own pace.
This project benefited from mathematical discussions with many people, especially (but not limited to!): Patrick Allen, George Boxer, Matthew Emerton, Haruzo Hida, Alex Horawa, and Gal Porat.
Thanks also to Patrick Allen for comments on a first version of this post-thesis article.
% Frank
% Matt for second advisors
% Patrick for reading
% other math conversations

% UChicago math grad students
The author's time at the University of Chicago was improved immeasurably by many other graduate students.
Thanks to Mathilde and Olivier, for being such constant companions on my grad school journey; to my academic siblings Shiva, Joel, Noah for providing a great environment to learn number theory in; and to the crossword crew for the many joyful lunch hours hanging out and solving puzzles.

% Mathcamp found family, and actual IRL family
The author is incredibly grateful to the Canada/USA Mathcamp community for creating such a magical place, and welcoming the author into it over the past 3 summers.
The author's family has been incredibly supportive throughout the author's time in graduate school, and this journey would not have been possible without them.

Lastly the author wishes to acknowledge the support he received during his graduate studies from Canada's National Sciences and Engineering Research Council.

\section{Modular forms}\label{sec:modular_forms}

In this section we collect some facts about elliptic modular eigenforms, their Hecke fields, and bounds on their Hecke eigenvalues.
Everything in this section is either already known or easily deduced from known results, we are simply collecting the key facts for our arguments to have them in one place.

\subsection{Galois conjugates}\label{sec:galois_conjugates}

Acting on the space $S_{k}(N, \epsilon; \C)$ of cuspidal modular forms of weight $k$, level $\Gamma_{0}(N)$, Nebentypus character $\epsilon$, with complex coefficients is a commutative algebra $H_{k}(N, \epsilon; \C)$, called the Hecke algebra.
The generators of this algebra and their action on $q$-expansions are:
\begin{itemize}
\item for each prime $\ell \nmid N$ a Hecke operator $T_{\ell}$ acting by
\[
T_{\ell}\left(\sum_{n=1}^{\infty} a_{n}q^{n} \right) = \sum_{n=1}^{\infty} a_{n \ell} q^{n} + \epsilon(\ell) \ell^{k-1} \sum_{n=1}^{\infty} a_{n}q^{n \ell},
\]
\item for each prime $\ell \nmid N$ a diamond operator $S_{\ell}$ acting by 
\[
S_{\ell}\left(\sum_{n=1}^{\infty} a_{n}q^{n} \right) = \epsilon(\ell) \sum_{n=1}^{\infty} a_{n}q^{n},
\]
\item for each prime $\ell | N$ an Atkin-Lehner operator $U_{\ell}$ acting by
\[
U_{\ell}\left(\sum_{n=1}^{\infty} a_{n}q^{n} \right) = \sum_{n=1}^{\infty} a_{n \ell} q^{n}.
\]
\end{itemize}

Throughout this article we will mostly work with Hecke algebras rather than directly with spaces of forms.
Fundamentally the two are equivalent thanks to the following proposition.

\begin{proposition}\label{prop:hecke_algebra_pairing}
The pairing
\begin{align*}
\begin{array}{rclll}
H_{k}(N, \epsilon; \C) & \times & S_{k}(N, \epsilon; \C) 	& \to & \C \\
(T&,&f)		& \mapsto & a_{1}(T(f))
\end{array}
\end{align*}
is a perfect pairing of complex vector spaces.
\end{proposition}

% Working with the larger space of forms for the group $\Gamma_{1}(N)$, we can consider the $\Q$ or $\Z$-algebras generated by these same Hecke operators.
% As $H_{k}(\Gamma_{1}(N); \C)$ is reduced we get that $H_{k}(\Gamma_{1}(N); \Q)$ must be a product of finite extensions of $\Q$.
% Tracing back through the duality between Hecke algebras and spaces of forms, we see that $S_{k}(\Gamma_{1}(N); \C)$ admits a basis of eigenvectors of the action of $H_{k}(\Gamma_{1}(N); \C)$, each corresponding to a homomorphism $H_{k}(\Gamma_{1}(N); \C)$ factoring through one of the number field components of the Hecke algebra.
% If we scale such a basis vector $f$ so that $a_{1}(f) = 1$, we call $f$ a normalized eigenform.
% Note that $S_{k}(N, \epsilon; \C)$ is just the subspace of $S_{k}(\Gamma_{1}(N); \C)$ where the quotient $\Gamma_{0}(N)/\Gamma_{1}(N) \cong (\Z/N\Z)^{\times}$ acts by $\epsilon$.

Thanks to this perfect pairing, we have that $S_{k}(N, \epsilon; \C)$ admits a basis of (simultaneous) eigenvectors for the action of $H_{k}(N, \epsilon; \C)$.
If $f$ is such an eigenvector, we have that there is a homomorphism $\psi(f): H_{k}(N, \epsilon; \C) \to \C$ sending $T_{\ell}$ (resp. $S_{\ell}$ or $U_{\ell}$) to its eigenvalue on $f$.
If we scale the eigenvector $f$ so that $a_{1}(f) = 1$, we call $f$ a normalized eigenform, and we have that $T_{\ell}(f) = a_{\ell}(f)$, the $\ell$-th Fourier coefficient of $f$.

Moreover we can consider the $\Q$ or $\Z$-algebra generated by these same operators as a Hecke algebra $H_{k}(N, \epsilon; \Q)$ or $H_{k}(N, \epsilon; \Z)$.
Since the integral Hecke algebra is finitely generated over $\Z$ (it is a sub-algebra of the finitely generated $\End_{\Z}(S_{k}(N, \epsilon; \Z))$) each such homomorphism $\psi(f)$ must have image in the ring of integers of a finite extension of $\Q$.
With this structure in place we can define Galois conjugation on these spaces of modular forms.
Given a normalized eigenform $f \in S_{k}(N, \epsilon; \C)$ and an element $\sigma$ of the absolute Galois group $G_{\Q}$ we define the Galois conjugate $f^{\sigma}$ by letting $\psi_{f}: H_{k}(N, \epsilon; \Z) \to \C$ be the homomorphism corresponding to $f$ by the duality between Hecke algebras and modular forms, and letting $f^{\sigma}$ be the normalized eigenform corresponding to $\sigma \circ \psi_{f}: H_{k}(N, \epsilon^{\sigma}; \Z) \to \C$.
Note that this is well-defined as we can either choose an extension of $\sigma$ to all of $\C$ or observe that the image of $\psi_{f}$ lands in a number field.

\begin{remark}
We could define the Galois action on modular forms directly on $q$-expansions by acting with any field automorphism of $\C/\Q$, but with that definition is it entirely unclear why the Galois conjugate of a modular form should still be a modular form!
The given definition makes it clear that Galois conjugate of a Hecke eigenform is still an eigenform.
\end{remark}

See section 6.5 of \cite{diamond_shurman} for a more in depth discussion of Galois conjugates of eigenforms.

\subsection{Hecke fields}\label{sec:hecke_fields}
% duality between Hecke algebra and space of forms
% Galois action on space of forms
% need to define character field and Hecke field
% need to define Galois conjugates
% need to know that they're in the right same space of forms
% also Hecke fields for CM forms

\begin{definition}
Given a normalized eigenform $f \in S_{k}(N, \epsilon; \C)$ we define its character field
\[
\Q(\epsilon) = \Q(\{\epsilon(x): x \in (\Z/N\Z)^{\times}\})
\]
and its Hecke field
\[
\Q(f) = \Q(\{a_{n}(f): n \in \N\}).
\]
\end{definition}

\begin{remark}
We make several remarks about character and Hecke fields of modular forms.
First off, the relation between Fourier coefficients
\[
a_{\ell^{2}}(f) = a_{\ell}(f)^{2} - \epsilon(\ell)\ell^{k-1} a_{\ell}(f)
\]
for almost all primes $\ell$ shows that $\Q(f) \supseteq \Q(\epsilon)$.

Second, since we know that the Fourier coefficient $a_{\ell}(f)$ is equal to the eigenvalue of the operator $T_{\ell}$ acting on $f$ (similarly $\epsilon(\ell)$ is the eigenvalue of $S_{\ell}$ acting on $f$) we have that $\Q(f)$ is the image of the homomorphism $H_{k}(N, \epsilon; \Q(\epsilon)) \to \C$ sending $T_{\ell}$ to $a_{\ell}(f)$.
Since the Hecke algebra is finitely integrally generated we have that $\Q(f)$ is a finitely generated extension of $\Q$.
Moreover the finiteness of the Hecke algebra shows that each $a_{\ell}(f)$ must be algebraic (integral even) and so we conclude that $\Q(f)$ is of finite degree over $\Q$.
\end{remark}

We record here the important principle that the degree of the Hecke field tells us the number of Galois conjugates of a form.

\begin{lemma}\label{lem:hecke_field_galois_conjugates}
Let $f \in S_{k}(N, \epsilon; \C)$ be a normalized eigenform.
Then the number of Galois conjugates of $f$ over $\Q$ is equal to the degree $[\Q(f):\Q]$ of the Hecke field of $f$ over $\Q$.
\end{lemma}
\begin{proof}
This is more generally just a fact about algebraic field extensions.
The degree $[\Q(f):\Q]$ is equal to the size of the orbit of the generating set $\{a_{n}(f): n \in \N\}$ under the action of the absolute Galois group of $\Q$.
But we also have that the size of this orbit is the number of Galois conjugates of $f$ itself, since for $\sigma, \tau \in G_{\Q}$ we have that $\sigma(a_{n}(f)) = \tau(a_{n}(f))$ for all $n$ if and only if $f^{\sigma} = f^{\tau}$.
\end{proof}

\subsection{Modular forms with complex multiplication}\label{sec:cm_forms}

% the easiest definition is that form has CM if a_p = 0 when $chi(p) = -1$
% by cebotarev this shows that Galois rep is twist, hence induced
% better definition is that form is newform attached to CM theta series?
% ribet 

\begin{definition}
Let $E$ be an imaginary quadratic field.
We say that a modular form $f$ has complex multiplication (or CM for short) by $E$ if $a_{p}(f) = 0$ whenever $p$ is inert in the extension $E/\Q$.
We say that a modular form has CM if it has CM by some imaginary quadratic field.
\end{definition}

A good reference for basic facts about modular forms with complex multiplication is Sections 3 and 4 of \cite{ribet_nebentypus}.
Of interest to us is the fact that eigenforms with CM can be constructed using algebraic Hecke characters, which we recall here.
Let $E$ be an imaginary quadratic field with a chosen embedding $\sigma: E \to \C$, and let $\psi$ be an algebraic Hecke character of infinity-type $\sigma^{k-1}$ and conductor $\frakm$.
Given such a Hecke character one can construct a weight $k$ eigenform given by the series
\begin{align*}
g = \sum_{\fraka} \psi(\fraka) q^{\Norm^{E}_{\Q}(\fraka)}
\end{align*}
where the sum is over all integral ideals $\fraka$ of $\calO_{E}$ which have $(\fraka, \frakm) = 1$.
The modular form $g$ is called the $\theta$-series attached to $\psi$.
The eigenform $g$ thus constructed has level $DM$ where $D$ is the discriminant of $E/\Q$ and $M = \Norm^{E}_{\Q}(\frakm)$, and character $\epsilon = \chi \eta$ where $\chi$ is the quadratic Dirichlet character attached to $E$ and $\eta$ is the ``finite order'' part of $\psi$ on the integers, given by $\eta(n) = \psi((n))/\sigma(n)^{k-1}$ for $n \in \Z$.
Note that it is immediate from this definition that $g$ has CM by $E$; if $p$ is inert in $E/\Q$ then there are no ideals of $\calO_{E}$ having norm $p$, so $a_{p}(g) = \sum_{\fraka, \Norm^{E}_{\Q}(\fraka) = p} \psi(\fraka) = 0$.
This construction is studied in Section 3 of \cite{ribet_nebentypus}, and Section 4 of \cite{ribet_nebentypus} deals with basic facts about Galois representations attached to CM eigenforms.

Of fundamental importance for our method is the characterization of CM families by the arithmetic complexity of their Hecke fields.
This characterization due to Hida is recalled in \cref{sec:bounded}.
The key philosophy is that Hecke fields attached to CM eigenforms are much simpler than those attached to non-CM eigenforms.
We begin this study here with a description of the Hecke field of a CM eigenform.

\begin{lemma}\label{lem:cm_hecke_fields}
Suppose that $f \in S_{k}(N, \epsilon; \C)$ be a normalized eigenform with CM by the imaginary quadratic field $E$.
Suppose that $f$ is realized as a theta series by an algebraic Hecke character $\psi$ of conductor $\frakm$, and let $M = \Norm^{E}_{\Q}(\frakm)$.
Let $h$ be the class number of $E$.
Then there are elements $a_{1}, \ldots, a_{h}$ in $E$ such that
\[
\Q(f) \subseteq E(\mu_{h \cdot M}, a_{1}^{1/h}, \ldots, a_{h}^{1/h}).
\]
In particular, the degree of the Hecke field over $\Q$ is bounded solely in terms of the CM field $E$ and the conductor $\frakm$ of the character $\psi$ (the degree is at most $2h^{3}M$ over $\Q$).
\end{lemma}
\begin{proof}
We know that our eigenform $f$ has $q$-expansion given by
\[
f = \sum_{\fraka} \psi(\fraka) q^{\Norm^{E}_{\Q}(\fraka)}
\]
where the sum is over all integral ideals $\fraka$ of $\calO_{E}$ which have $(\fraka, \frakm) = 1$.
We certainly have that $\Q(f) \subseteq \Q(\{\psi(\fraka)\})$, so it suffices to understand the field of definition of $\psi$.

If $(a)$ is a principal ideal with generator $a \equiv 1 \mod{\frakm}$, we have that $\psi((a)) = a^{k - 1}$.
If the generator $a$ is not necessarily $\equiv 1 \mod{\frakm}$, then we know that $a^{hM} \equiv 1 \mod{\frakm}$ and hence $\frac{\psi((a))}{a^{k-1}} \in \mu_{hM}$.

Let $\fraka_{1}, \ldots, \fraka_{h}$ be a complete set of representatives of the ideal class group of $E$.
For each $i$ we have that $\psi(\fraka_{i}^{h}) \in E$ since $\fraka_{i}^{h}$ is principal and $\psi$ necessarily sends principal ideals to elements of $E$ (using that the infinity-type of $\psi$ is $\sigma^{k-1}$ for $\sigma:E \to \C$).
Let $a_{i}$ be an element of $E$ such that $\psi(\fraka_{i}^{h}) = a_{i}$.
We know that every ideal $\fraka$ is equal to a principal ideal times one of our representatives $\fraka_{i}$, hence we have that $\psi(\fraka)$ must be a root of unity of order $hM$ times $a_{i}^{1/h}$.
\end{proof}

This idea of the uniformity of Hecke fields of CM forms will be studied further in \cref{sec:hida_cm}, where we show that there is a uniform description of the Hecke fields of all forms in an ordinary family which has CM.

\subsection{Archimedean bounds}\label{sec:arch_bounds}
% only need in the principal series case
% to give a hilbert proof here need to find a thing to cite which deals with GL(2, F) rather than just GL(2, \Q)

In this section we state bounds on the (Archimedean) absolute value of Hecke eigenvalues of modular forms.

\begin{theorem}[The Ramanujan conjecture, due to Deligne \cite{deligne_weil_i}]\label{thm:ramanujan}
Let $f \in S_{k}(N, \epsilon; \C)$ be a normalized eigenform.
Then for all primes $\ell$ we have that
\[
|a_{\ell}(f)|_{\C} \leq 2 \ell^{\frac{k-1}{2}}.
\]
\end{theorem}

\begin{remark}
Let $\rho_{f, p}$ be the $p$-adic Galois representation attached to $f$.
Let $\alpha_{\ell}, \beta_{\ell}$ be the eigenvalues of $\rho_{f, p}(\Frob_{\ell})$; these are the roots of $x^{2} - a_{\ell}(f)x + \epsilon(\ell)\ell^{k-1}$.
The bound of \cref{thm:ramanujan} then gives that $\alpha_{\ell}, \beta_{\ell}$ are $\ell$-\emph{Weil numbers}, in other words their complex absolute values are exactly $p^{\frac{k-1}{2}}$.
This fact is key to the characterization of CM ordinary families due to Hida that we employ in \cref{sec:bounded}.
\end{remark}

\Cref{thm:ramanujan} also holds in weight one, where it is due to Deligne--Serre in \cite{deligne_serre}.
However, we wish to avoid using it in low weight in order to provide a proof technique which has the possibility of applying to the case of Hilbert modular forms of partial weight one.
To get around this we will only use a weaker bound in the case $k = 1$, the analog of which is comparatively easy to establish for all forms.
Though this weaker bound can doubtless be extracted from the literature, we provide a proof here.
We note that all our method requires is an upper bound on $|a_{\ell}(f)|_{\C}$ which is independent of $f$ (but may depend on the weight $k$), so the exact exponent $\frac{k}{2} + 1$ that appears in our bound is unimportant.

\begin{theorem}\label{thm:trivial_coefficient_bound}
Let $f \in S_{k}(N, \epsilon; \C)$ be a normalized eigenform.
Then for all primes $\ell \nmid N$ we have that
\[
|a_{\ell}(f)|_{\C} \leq 2 \ell^{\frac{k}{2}+1}.
\]
\end{theorem}
\begin{proof}
The proof is a combination of two facts.
First, for any cusp form $f(q) = \sum_{n=1}^{\infty} a_{n}q^{n}$ of weight $k$ we know that there is some constant $C$ such that $|a_{n}|_{\C} \leq C n^{\frac{k}{2}}$ and so the $L$-series
\[
L(s, f) = \sum_{n=1}^{\infty} \frac{a_{n}}{n^{s}}
\]
associated to $f$ converges absolutely in the right half-plane $\Realpart(s) > \frac{k}{2} + 1$.
This bound on coefficients does not suffice for our purposes, as the constant $C$ may depend on the form $f$.
The second fact that we need is that since $f$ is a normalized eigenform its $L$-series admits an Euler product expansion,
\[
L(s, f) = \prod_{p} \frac{1}{1 - a_{p}p^{-s} + \epsilon(p)p^{k-1}p^{-2s}}.
\]
Both of these facts can be found in section 5.9 of \cite{diamond_shurman}.

Consider the series $L_{p}(s, f) = \sum_{r=0}^{\infty} \frac{a_{p^{r}}}{p^{-rs}}$.
On one hand we know that this series converges absolutely on the right half-plane $\Realpart(s) > \frac{k}{2} + 1$, since it is just a sum over fewer terms of the series $L(s, f)$.
We also know that this is equal to the geometric series $\frac{1}{1 - a_{p}p^{-s} + \epsilon(p)p^{k-1}p^{-2s}}$ after some rearranging of terms, and a geometric series $\frac{1}{1-x}$ converges exactly when $|x| < 1$.
So we get that
\[
|a_{p}p^{-s} - \epsilon(p)p^{k-1}p^{-2s}|_{\C} < 1
\]
whenever $\Realpart(s) > \frac{k}{2} + 1$.
Taking a limit as $\Realpart(s) \to \frac{k}{2} + 1$ and juggling the inequality around yields the desired bound
\[
|a_{p}|_{\C} \leq 2 p^{\frac{k}{2} + 1}.
\]
\end{proof}

\subsection{Non-Archimedean bounds}\label{sec:non-arch_bounds}
% slope bounds
% proof 1: via the automorphic rep
% to give a hilbert proof here is just more complicated automorphic setup
In this section we discuss non-Archimedean bounds on the Hecke eigenvalues of modular forms.
The main result is a classical bound on the $U_{p}$ eigenvalue of classical forms: if $f$ has weight $k$ and its $U_{p}$ eigenvalue is non-zero, then that eigenvalue has valuation bounded between $0$ and $k-1$.
This valuation is often referred to as the slope of $f$.
This bound appears throughout the literature, and some cases can be easily proved using the Galois representations attached to eigenforms, but we provide a proof using the automorphic representations attached to eigenforms that covers all cases of interest.
Following that we derive a key consequence for weight one forms.

\begin{theorem}\label{thm:slope_bound}
Let $f \in S_{k}(N, \epsilon; \C_{p})$ be a normalized eigenform.
Suppose that $p$ divides $N$, so we have that the $U_{p}$ eigenvalue of $f$ is $a_{p}(f)$.
If $a_{p}(f) \neq 0$ then
\[
0 \leq \ord_{p}(a_{p}(f)) \leq k-1.
\]
\end{theorem}
\begin{proof}
We know from the integrality of the Hecke algebra that $a_{p}(f)$ will be integral, i.e. $\ord_{p}(a_{p}(f)) \geq 0$.
Let $\pi_{f}$ be the automorphic representation attached to $f$.
We choose to normalize $\pi_{f}$ so that the $U_{p}$ eigenvalue of the modular form $f$ is $\sqrt{p}$ times the $U_{p}$ eigenvalue of $\pi_{f, p}$.
Thus we wish to show that the $U_{p}$ eigenvalue of $\pi_{f, p}$ has valuation bounded above by $k - \frac{3}{2}$.
Under the assumption that the $U_{p}$ eigenvalue is non-zero, there are three possible cases for what $\pi_{f, p}$ can be: an irreducible principal series representation with both characters unramified, an irreducible principal series representation with one character unramified, or an unramified twist of the Steinberg representation.
We treat each case separately to establish the upper bound.

Case 1: $\pi_{f, p}$ is an irreducible principal series representation $PS(\chi_{1}, \chi_{2})$ where both characters $\chi_{i}$ are unramified.
In this case we know that the $U_{p}$ eigenvalue is either $\chi_{1}(p)$ or $\chi_{2}(p)$.
Since the central character of $\pi_{f}$ has weight $k-2$, we know that the product $\chi_{1}(p)\chi_{2}(p)$ has valuation $k-2$.
Since we know that $\ord_{p}(\chi_{i}(p)) \geq -\frac{1}{2}$, we get that each must have valuation at most $k - 2 + \frac{1}{2} = k - \frac{3}{2}$.
Thus the $U_{p}$ eigenvalue of $\pi_{f, p}$ has valuation at most $k - \frac{3}{2}$.

Case 2: $\pi_{f, p}$ is an irreducible principal series representation $PS(\chi_{1}, \chi_{2})$ where only $\chi_{1}$ is unramified.
Let $\alpha_{1} = \chi_{1}(p)$, which is the $U_{p}$ eigenvalue of $\pi_{f, p}$, hence it has valuation at least $-\frac{1}{2}$.
Let $\chi$ be the character of $\prod_{\ell} \Z_{\ell}^{\times}$ which is equal to $\chi_{2}$ on the $p$-component and trivial on all others; we can view this as a finite order Dirichlet character.
Take $g$ to be the eigenform $f \otimes \chi^{-1}$.
Thus we have that $\pi_{g, p} = \pi_{f, p} \otimes \chi^{-1}|_{\Z_{p}^{\times}}$, which is the principal series representation $PS(\chi_{1}\chi^{-1}, \chi_{2}\chi^{-1})$.
Note that our choice of $\chi$ means that $\chi_{1}\chi^{-1}$ is ramified and $\chi_{2}\chi^{-1}$ is unramified.
Let $\alpha_{2} = \chi_{2}(p)$, which is the $U_{p}$ eigenvalue of $g$, hence it has valuation at least $-\frac{1}{2}$.
Since we've only twisted $f$ by a finite order character that weight of the central character remains unchanged.
From this we conclude that $\ord_{p}(\alpha_{1}\alpha_{2}) \leq k - 2$, and since each has valuation at least $-\frac{1}{2}$ we conclude that each has valuation at most $k - \frac{3}{2}$.

Case 3: $\pi_{f, p}$ is an unramified twist of the Steinberg representation $S(\chi)$.
In this case the $U_{p}$ eigenvalue of $\pi_{f, p}$ is $\chi(p)$.
We know that the central character evaluated at $p$ is equal to $p \chi(p)^{2}$.
Since this must have valuation equal to $k - 2$, we see that $\chi(p)$ has valuation equal to $\frac{k-3}{2}$ which is certainly less than $k - \frac{3}{2}$.
Note in particular that this cannot occur when $k = 1$ since $\frac{k-3}{2} = -1$ is less than $-\frac{1}{2}$, which we already know to be a lower bound on the valuation.

Thus in all cases we have the desired bounds on the $U_{p}$ eigenvalue of $\pi_{f, p}$, which gives us the desired bounds on the $U_{p}$ eigenvalue of $f$ itself.
\end{proof}

Using this bound on the valuation of the $U_{p}$ eigenvalue of an eigenform, we prove the ``automatically ordinary'' property for weight one eigenforms alluded to in \cref{sec:strategy}.
This is simply a matter of applying the bound from \cref{thm:slope_bound} in the case $k = 1$.

\begin{corollary}\label{cor:weight_one_conjugates}
Let $f \in S_{1}(N, \epsilon; \C_{p})$ be a normalized eigenform of weight one.
If $a_{p}(f) \neq 0$ then
\[
\ord_{p}(a_{p}(f^{\sigma})) = 0
\]
for all $\sigma \in G_{\Q}$.
\end{corollary}
\begin{proof}
A Galois conjugate $f^{\sigma}$ of $f$ will be a normalized eigenform in the space $S_{1}(N, \epsilon^{\sigma}; \C_{p})$.
In particular \cref{thm:slope_bound} still applies to $f^{\sigma}$, since $a_{p}(f^{\sigma}) = \sigma(a_{p}(f)) \neq 0$.
So we conclude that 
\[
0 \leq \ord_{p}(a_{p}(f^{\sigma})) \leq 1 - 1 = 0.
\]
\end{proof}
\section{Families of modular forms}\label{sec:hida}

\subsection{Ordinary families of elliptic modular forms}\label{sec:hida_elliptic}

In this section we summarize the elements of Hida's theory of ordinary families of elliptic modular forms which we will use in later sections.
The main idea is that for any space of forms with level divisible by $p$ we have an action of the $U_{p}$ operator.
Hida's key realization was that the $U_{p}$-ordinary subspace of a space of modular forms has bounded dimension as we vary the weight and Nebentypus character.
As a consequence of this the $U_{p}$-ordinary subspaces of forms in a fixed tame level can be interpolated into a single family, finite over a weight space $\Lambda$ parametrizing the weight-character (viewed as a $p$-adic character of $\Z_{p}^{\times}$).
Whenever we discuss ordinarity (of a space of modular forms, or Hecke algebra, etc.) from now we always mean ordinarity with respect to the $U_{p}$ operator, so we will say ``ordinary'' rather than ``$U_{p}$-ordinary''.
For elliptic forms all the statements we need can be found in Hida's original papers \cite{hida_iwasawa_modules} and \cite{hida_galois_representations} together with Wiles' work on Galois representations attached to ordinary eigenforms \cite{wiles_ordinary}.

We work with $\Lambda$-adic Hecke algebras as our main objects.
We fix the following notation for use in this section.
\begin{itemize}
\item An odd prime number $p$.
\item A positive integer $N$ coprime to $p$, which will be the prime-to-$p$ part of the level of our forms.
\item $K$ a finite extension of $\Q_{p}$, with ring of integers $\calO_{K}$.
\item $\Lambda = \calO_{K}\llbracket T \rrbracket$, the ring of formal power series in one variable over $\calO_{K}$.
\item For any positive integer $k$ and $p$-power root of unity $\zeta$, let $P_{k, \zeta}$ be the kernel of the homomorphism
\begin{align*}
\Lambda	& \to \Qbar_{p} \\
T 		& \mapsto \zeta(1 + p)^{k-1} - 1.
\end{align*}
\end{itemize}

We view $\Lambda$ as the $\calO_{K}$ group ring of the torsion-free part of the Galois group 
\[
\Gal(\Q(\mu_{p^{\infty}})/\Q) \cong \Z_{p}^{\times}.
\]
The torsion-free part is $(1 + p)\Z_{p}$; the isomorphism with $\calO_{k}\llbracket T \rrbracket$ is realized by sending $1 + p$ to $T$.
While we could view all of our Hecke algebras as living over the larger group ring $\calO_{K}\llbracket \Z_{p}^{\times} \rrbracket$, the $(\Z/p\Z)^{\times}$ part of the character plays no role in our arguments so we will work solely with Hecke algebras as $\Lambda$-modules.
Given the above setup, Hida's theory asserts the existence of a ``universal'' ordinary Hecke algebra.

\begin{theorem}\label{thm:hida_theory_main_statement}
There exists a Hecke algebra $\mathbf{H}^{\ord}(N; \calO_{K})$ with the following properties.
\begin{enumerate}
\item $\mathbf{H}^{\ord}(N; \calO_{K})$ is a finitely generated free $\Lambda$-module.
% \item $\mathbf{H}^{\ord}(N; \calO_{K})$ is reduced.
\item (Base Change) If $L$ is a finite extension of $K$, with ring of integers $\calO_{L}$, we have that
\[
\mathbf{H}^{\ord}(N; \calO_{L}) \cong \mathbf{H}^{\ord}(N; \calO_{K}) \otimes_{\calO_{K}} \calO_{L}.
\]
\item $\mathbf{H}^{\ord}(N; \calO_{K})$ is generated as a $\Lambda$-module by a collection of elements $T_{\ell}$, $S_{\ell}$ for each prime $\ell\nmid Np$, and elements $U_{\ell}$ for each prime $\ell | Np$.
\item (Control Theorem) Let $k \geq 2$ be an integer, and $\zeta$ a $p^{r}$-th root of unity for some $r \geq 0$.
Suppose that $K$ is large enough to contain $\zeta$.
Let $\epsilon: (1+p)\Z_{p} \to \calO_{K}^{\times}$ be the character taking $1+p$ to $\zeta$.
Then the natural map
\[
\mathbf{H}^{\ord}(N; \calO_{K})/P_{k, \zeta} \mathbf{H}^{\ord}(N; \calO_{K}) \overset{\cong}{\to} H^{\ord}_{k}(\Gamma_{1}(Np) \cap \Gamma_{0}(p^{r}), \epsilon; \calO_{K})
\]
sending the abstract elements $T_{\ell}, S_{\ell}, U_{\ell}$ on the left to the equivalently named Hecke operators on the right is an isomorphism of $\Lambda$-modules.
In other words $\mathbf{H}^{\ord}(N; \calO_{K})$ interpolates the ordinary subspaces of all spaces of forms with prime to $p$ level $\Gamma_{1}(N)$ and weight $k \geq 2$.
\end{enumerate}
\end{theorem}
\begin{proof}
The fact that $\mathbf{H}^{\ord}(N; \calO_{K})$ is finite free over $\Lambda$ is Theorem 3.1 of \cite{hida_iwasawa_modules}.
The control theorem is Theorem 1.2 of \cite{hida_galois_representations}.
The other statements are all consequences of the definition of $\mathbf{H}^{\ord}(N; \calO_{K})$, and can be found in \cite{hida_iwasawa_modules}.
\end{proof}

There are two constructions of this universal ordinary Hecke algebra $\mathbf{H}^{\ord}(N; \calO_{K})$.
The first, using Katz's theory of geometric $p$-adic modular forms, appears in \cite{hida_iwasawa_modules}.
The second, based on Betti cohomology of modular curves and group cohomology of congruence subgroups of $\SL_{2}(\Z)$, appears in \cite{hida_galois_representations}.
We refer to these approaches as the \emph{geometric} and \emph{cohomological} approaches to Hida theory.
Each approach has benefits and drawbacks, and both are necessary in order to develop all facets of the theory which we use in this work.
The geometric approach is crucial to understanding how forms of weight one fit into $\mathbf{H}^{\ord}(N; \calO_{K})$, a topic which we explore in \cref{sec:hida_weight_one}.
If one wants freeness over $\Lambda$ of $\mathbf{H}^{\ord}(N; \calO_{K})$ rather than just torsion-freeness this is provided only by the geometric approach.
A downside of the geometric approach is that, at least in Hida's original work, it only deals with the case $r = 0$ of the control theorem, i.e. forms with trivial Nebentypus character.
While this is sufficient to uniquely determine $\mathbf{H}^{\ord}(N; \calO_{K})$ it is not enough for our application, as we will need to specialize at infinitely many different Nebentypus characters in a single weight.
The cohomological approach is comparatively simpler as it does not require the algebraic geometry machinery of the geometric approach.
Proving the control theorem for all characters $\epsilon$ is much more straightforward under the cohomological framework than the geometric one.
The downsides of the cohomological approach are that it only produces torsion-freeness of $\mathbf{H}^{\ord}(N; \calO_{K})$ over $\Lambda$ (as opposed to freeness) and that it gives no information about the weight one specializations of the Hecke algebra.

We turn now to a discussion of ``components'' of $\mathbf{H}^{\ord}(N; \calO_{K})$.
As we are interested in maps from $\mathbf{H}^{\ord}(N; \calO_{K})$ to rings of integers, for the moment we work with its maximal reduced quotient $\mathbf{H}^{\ord}(N; \calO_{K})^{\red}$.
We know that $\mathbf{H}^{\ord}(N; \calO_{K})^{\red} \otimes_{\Lambda} \Frac(\Lambda)$ is a product of finite field extensions of $\Frac(\Lambda)$.
Say that
\[
\mathbf{H}^{\ord}(N; \calO_{K})^{\red} \otimes_{\Lambda} \Frac(\Lambda) \cong \prod_{i=1}^{n} \Frac(\mathbf{I}_{i})
\]
where $\mathbf{I}_{i}$ is an integral extension of $\Lambda$.
We then have that 
\[
\mathbf{H}^{\ord}(N; \calO_{K})^{\red} \hookrightarrow \prod_{i=1}^{n} \mathbf{I}_{i}
\]
where each projection map $\mathbf{H}^{\ord}(N; \calO_{K}) \to \mathbf{I}_{i}$ is surjective, although the total map need not be surjective.
The $\mathbf{I}_{i}$ are the ``components'' of $\mathbf{H}^{\ord}(N; \calO_{K})$ (technically it is more accurate to say that $\Spec(\mathbf{I}_{i})$ is a component of $\Spec(\mathbf{H}^{\ord}(N; \calO_{K}))$, though we won't use this point of view).
Note that if $\mathbf{I}$ is a component of $\mathbf{H}^{\ord}(N; \calO_{K})$ we have that $\mathbf{I} = \mathbf{H}^{\ord}(N; \calO_{K})/\frakP$ for some minimal prime ideal $\frakP$ of $\mathbf{H}^{\ord}(N; \calO_{K})$.

Given a normalized eigenform $f \in S_{k}^{\ord}(\Gamma_{1}(Np) \cap \Gamma_{0}(p^{r}), \epsilon; \calO_{L})$ for some weight $k \geq 2$, character $\epsilon$ of conductor $r$, and finite extension $L$ of $K$, we say that $f$ arises from $\mathbf{H}^{\ord}(N; \calO_{K})$ or $f$ arises as a specialization of $\mathbf{H}^{\ord}(N; \calO_{K})$.
If $f$ arises from $\mathbf{H}^{\ord}(N; \calO_{K})$ and $\mathbf{I}$ is a component of $\mathbf{H}^{\ord}(N; \calO_{K})$, we say that $f$ arises from $\mathbf{I}$ if the homomorphism $\mathbf{H}^{\ord}(N; \calO_{K}) \to \calO_{K}$ realizing the eigensystem of $f$ factors through the surjective map $\mathbf{H}^{\ord}(N; \calO_{K}) \to \mathbf{I}$.
Thus far we have not said anything about what happens when $k = 1$; treating the case of weight $k = 1$ is the focus of \cref{sec:hida_weight_one}.

With this notion of forms arising from components, we can state an important uniqueness property of the $\Lambda$-adic Hecke algebra.
\begin{theorem}\label{thm:hida_etale}
Let $f \in S_{k}^{\ord}(\Gamma_{1}(Np) \cap \Gamma_{0}(p^{r}), \epsilon; \calO_{K})$ be a normalized eigenform.
Then there is a unique component $\mathbf{I}$ of the $\Lambda$-adic Hecke algebra $\mathbf{H}^{\ord}(N; \calO_{K})$ such that $f$ arises from $\mathbf{I}$.
\end{theorem}
\begin{proof}
This is Corollary 1.5 of \cite{hida_galois_representations}.
\end{proof}

\subsection{Galois representations attached to ordinary families}

Each of the normalized eigenforms which $\mathbf{H}^{\ord}(N; \calO_{K})$ interpolates has an attached $2$-dimensional $p$-adic Galois representation.
It should thus not be surprising that these Galois representations also interpolate into a single $\Lambda$-adic Galois representation.
These $\Lambda$-adic representations were first studied by Hida in \cite{hida_galois_representations} and Wiles in \cite{wiles_ordinary}.
We record here the minimal properties of these representations that we use in later sections.

\begin{theorem}\label{thm:hida_theory_galois_representation}
Suppose that $\mathbf{I}$ is a reduced, irreducible component of $\mathbf{H}^{\ord}(N; \calO_{K})$.
Then there exists a continuous $2$-dimensional Galois representation
\[
\rho_{\mathbf{I}}: G_{\Q} \to \GL_{2}(\Frac(\mathbf{I}))
\]
which has the following properties.
\begin{enumerate}
\item $\rho_{\mathbf{I}}$ is absolutely irreducible.
\item $\rho_{\mathbf{I}}$ is unramified away from $Np$, and the characteristic polynomial of a Frobenius element at a prime $\ell \nmid Np$
\[
X^{2} - T_{\ell} X - \ell S_{\ell}.
\]
\item When restricted to a decomposition group at $p$, $\rho_{\mathbf{I}}$ is of the form
\[
\rho_{\mathbf{I}}|_{G_{\Q_{p}}} \cong \begin{bmatrix} \ast & \ast \\ 0 & \lambda \\ \end{bmatrix}
\]
where $\lambda: G_{\Q_{p}} \to \mathbf{I}^{\times}$ is the unramified character sending $\Frob_{p}$ to $U_{p}$.
% \todo{get the right top corner formula (p-adic cyclo times $U_{p}^{-1}$) note that $U_{p}$ is invertible because ordinary}
\item For almost all primes $P$ of $\mathbf{I}$, the representation $\rho_{\mathbf{I}}$ can be taken to have values in the localization $\mathbf{I}_{P}$.
In particular for almost all normalized eigenforms $f$ arising from $\mathbf{I}$ we have that the $p$-adic Galois representation $\rho_{f, p}$ attached to $f$ is equal to the composition of $\rho_{\mathbf{I}}: G_{\Q} \to \GL_{2}(\mathbf{I}_{P})$ combined with the quotient map $\GL_{2}(I_{P}) \to \GL_{2}(I_{P}/PI_{P})$ for some prime $P$ of $\mathbf{I}$.
\end{enumerate}
\end{theorem}
\begin{proof}
This is Theorem 2.1 of \cite{hida_galois_representations}.
\end{proof}

\subsection{Weight one forms in ordinary families}\label{sec:hida_weight_one}

While Hida's articles are very precise about the specialization of ordinary families in weights $k \geq 2$, eigenforms of weight one are not discussed directly in these articles.
It does follow from Hida's first construction of ordinary families, using geometric $p$-adic modular forms, that every classical $p$-ordinary weight one eigenform arises as the specialization of an ordinary family.
As this is not obviously stated in the literature, we discuss this explicitly here, along with a key consequence for the Hecke fields of $p$-ordinary weight one eigenforms.

% Here's why weight one forms appear, using the language of $\Lambda$-adic forms.
% can also see it from the original geometric construction
\begin{proposition}\label{prop:hida_theory_weight_one}
Given a character $\epsilon: (1+p)\Z_{p} \to \mu_{p^{r}}$ sending $1+p$ to a generator $\zeta$ of $\mu_{p^{r}}$, there is a natural surjective homomorphism
\[
\mathbf{H}^{\ord}(N;\calO_{K})/P_{1, \zeta} \mathbf{H}^{\ord}(N;\calO_{K}) \twoheadrightarrow H^{\ord}_{1}(\Gamma_{1}(Np) \cap \Gamma_{0}(p^{r}), \epsilon; \calO_{K}).
\]
sending the abstract elements $T_{\ell}, S_{\ell}, U_{\ell}$ on the left to the equivalently named Hecke operators on the right.
Put differently, every $p$-ordinary weight one eigenform arises as the specialization of an ordinary family.
\end{proposition}
\begin{proof}
In Hida's first article \cite{hida_iwasawa_modules}, the universal ordinary Hecke algebra $\mathbf{H}^{\ord}(N; \calO_{K})$ is constructed as a limit of Hecke algebras acting on the spaces $S_{k}^{\ord}(\Gamma_{1}(N); K/\calO_{K})$.
These spaces of forms (or a suitable direct sum of these spaces allowing for divided congruences) are dense in the space $\mathcal{S}$ of all ordinary geometric $p$-adic modular forms, and so $\mathbf{H}^{\ord}(N; \calO_{K})$ can also be viewed as the Hecke algebra acting on this single large space of $p$-adic modular forms.
Any space of forms $S_{k}^{\ord}(\Gamma_{1}(Np^{r}); \calO_{K})$ can be viewed as a subspace of $\mathcal{S}$ by interpreting these classical forms of higher level as $p$-adic modular forms; in particular this holds for $k = 1$.
At the level of Hecke algebras, this means that we can realize the Hecke algebra of $S_{k}^{\ord}(\Gamma_{1}(Np^{r}); \calO_{K})$ as a quotient of the Hecke algebra on $\mathcal{S}$.
Decomposing the Hecke algebra on $S_{k}^{\ord}(\Gamma_{1}(Np^{r}); \calO_{K})$ as a direct sum corresponding to the various possible Nebentypus characters, we get the desired result.
\end{proof}

\begin{lemma}\label{lem:weight_one_hecke_field_bound}
Suppose that $f \in S_{1}(\Gamma_{1}(Np) \cap \Gamma_{0}(p^{r}), \epsilon; \calO_{K})$ is a classical eigenform of weight one arising as a specialization of $\mathbf{H}^{\ord}(N; \calO_{K})$.
Recall the finite extensions of $\Q$ defined using the Hecke eigenvalues of $f$:
\begin{align*}
\Q(\epsilon) 	& = \text{ the character field of $f$} \\
\Q(f)			& = \text{ the Hecke field of $f$}.
\end{align*}
Then we have that
\[
[\Q(f):\Q(\epsilon)] \leq \rank_{\Lambda}(\mathbf{H}^{\ord}(N; \calO_{K})).
\]
\end{lemma}
\begin{proof}
The degree $[\Q(f):\Q(\epsilon)]$ is equal to the number of distinct Galois conjugates of $f$ by the absolute Galois group $G_{\Q(\epsilon)}$ of $\Q(\epsilon)$.
Let us assume that our local coefficient field $K$ is large enough to contain $\Q(f)$ and all of its Galois conjugates.
Call these Galois conjugates $f_{1} = f, f_{2}, \ldots, f_{n}$.
Each $f_{i}$ is a classical weight one eigenform of the same level and character as $f$, i.e. each $f_{i} \in S_{1}(\Gamma_{1}(Np) \cap \Gamma_{0}(p^{r}), \epsilon; \calO_{K})$.

Crucially, we know that $f$ is ordinary since it is a specialization of an ordinary family $\mathbf{H}^{\ord}(N; \calO_{K})$.
As Galois conjugates of an eigenform with finite slope, each $f_{i}$ necessarily has finite slope.
But since the slope of a finite slope classical weight $k$ eigenform must be between $0$ and $k-1$ by \cref{thm:slope_bound}, we conclude that each $f_{i}$ is in fact ordinary, since $k - 1 = 0$ when $k = 1$.
So each $f_{i}$ is in the ordinary subspace $S_{1}^{\ord}(\Gamma_{1}(Np) \cap \Gamma_{0}(p^{r}), \epsilon; \calO_{K})$.
This entire space is a quotient of $\mathbf{H}^{\ord}(N; \calO_{K})/P_{1, \zeta} \mathbf{H}^{\ord}(N; \calO_{K})$ for some height one prime ideal $P_{1, \zeta}$ of $\Lambda$ by \cref{prop:hida_theory_weight_one}, so in total we have that
\begin{align*}
[\Q(f):\Q(\epsilon)]	& = \text{ the number of distinct Galois conjugates of $f$ by $G_{\Q(\epsilon)}$} \\
						& \leq \rank_{\calO_{K}}(S_{1}^{\ord}(\Gamma_{1}(Np) \cap \Gamma_{0}(p^{r}), \epsilon; \calO_{K}) \\
						& \leq \rank_{\calO_{K}}(\mathbf{H}^{\ord}(N; \calO_{K})/P_{1, \zeta} \mathbf{H}^{\ord}(N; \calO_{K})) \\
						& \leq \rank_{\Lambda}(\mathbf{H}^{\ord}(N; \calO_{K})).
\end{align*}
Note that the first inequality holds since distinct Galois conjugates of $f$ are linearly independent, as they lie in distinct eigenspaces for the action of the Hecke algebra.
\end{proof}

\begin{remark}
If we are willing to use the Ramanujan conjecture for classical weight one eigenforms, then it is likely that \cref{lem:weight_one_hecke_field_bound} already provides a sufficient input to prove our main result without appealing to the constructions of \cref{sec:rigidity} and \cref{sec:construction}.
The goal of \cref{sec:rigidity} and \cref{sec:construction} is to find a method by which the Hecke field bound of \cref{lem:weight_one_hecke_field_bound} can be propagated into regular weight, where \cref{thm:hida_cm_iff_hecke_field} may be applied.
We expect that \cref{thm:hida_cm_iff_hecke_field} can be adapted to require only that the forms in question satisfy the Ramanujan conjecture; see \cref{rem:hida_theorem_weight_one} for a discussion of adapting Hida's result to the weight one case.

Note that the Ramanujan conjecture \emph{is} known for weight one forms, having been proved by Deligne--Serre as a consequence of their construction of the Galois representations attached to weight one forms in \cite{deligne_serre}); our interest in finding a method which avoids using the Ramanujan conjecture is so that this strategy also applies to the case of partial weight one Hilbert modular forms, where the Ramanujan conjecture is still open.
\end{remark}

\begin{remark}
We remark that the principle encapsulated by \cref{lem:weight_one_hecke_field_bound} is unique to weight one.
For forms of weight $k \geq 2$ it is frequently the case that not all Galois conjugates of a given $p$-ordinary form are $p$-ordinary.
Indeed, one may think of Hida's characterization of CM families \cref{thm:hida_cm_iff_hecke_field} as saying that for non-CM ordinary eigenforms, the proportion of Galois conjugates which are also ordinary goes to $0$ as we increase the level.
\end{remark}

\subsection{Components with complex multiplication}\label{sec:hida_cm}

In this section we recall properties of the CM components of the $\Lambda$-adic ordinary Hecke algebra.
A good reference for these facts is Section 7 of \cite{hida_iwasawa_modules}.

% construct the CM components
We sketch the construction of CM components outlined by Hida in \cite{hida_iwasawa_modules}.
Let $E$ be an imaginary quadratic field in which our fixed prime $p$ splits as $(p) = \frakp \overline{\frakp}$.
Fix an integral ideal $\frakm$ of $E$ which is coprime to $\frakp$; this $\frakm$ will serve as the tame conductor of our CM components.
Let $W$ be the id\`{e}le class group of $E$ of conductor $\frakm\frakp^{\infty}$, that is
\[
W = \A_{E}^{\times}/\overline{U E_{\infty}^{\times} E^{\times}}
\]
where $U = \prod_{\ell \neq \frakp} U_{\ell}$, with $U_{\ell}$ the entire group of integral units of the completion $E_{\ell}$ if $\ell$ is coprime to $\frakm$, and $U_{\ell}$ being those integral units which are congruent to $1$ mod $\frakm$ if $\ell$ divides $\frakm$. 
The group $\Gamma = 1+p\Z_{p}$ injects into $U_{\frakp}$ which itself injects into $W$.
Moreover since the id\`{e}le class group of conductor $\frakm$ is finite we have that $\Gamma$ has finite index in $W$.

Let us assume that our coefficient ring $\calO_{K}$ is large enough to contains the values of all characters of the finite group $W/\Gamma$.
Define $A = \calO_{K}\llbracket W \rrbracket$ to be the $\calO_{K}$ group ring of $W$.
Then the inclusion $\Gamma \to W$ gives a map on group rings $\Lambda \to A$ which realizes $A$ as a finite free $\Lambda$-module.

For an algebraic Hecke character $\psi$ on $E$ of conductor $\frakm \frakp^{r}$ for some $r \geq 0$, we can view the $p$-adic avatar $\psi_{p}$ of $\psi$ as a continuous $p$-adic character of $W$.
Since our fixed prime $p$ splits in $E$, any eigenform $f$ with CM by $E$ will necessarily be $p$-ordinary.
If $\psi$ is an algebraic Hecke character inducing the CM eigenform $f_{\psi}$, we have that the map $A \to \calO_{K}$ corresponding to $\psi_{p}$ realizes the Hecke eigensystem of $f_{\psi}$ inside $\calO_{K}$.
Let $M = \Norm^{K}_{\Q}(\frakm)$, and let $-d$ be the discriminant of $E/\Q$.
In particular we have that
\[
a_{\ell}(f) = \begin{cases} 0 & \ell \text{ is inert in $E/\Q$} \\ \psi(\frakl) + \psi(\overline{\frakl}) & \ell \text{ splits as $\frakl \overline{\frakl}$ in $E/\Q$} \end{cases}
\]
for primes $\ell \nmid dMp$.
Since these quantities vary continuously with the character $\psi$, we can patch them together into a single map with coefficients in $A$.
Letting $\Psi: W \to A^{\times}$ be the tautological character, we have a map
\begin{align*}
\mathbf{H}^{\ord}(dM; \calO_{K}) 	& \to A \\
T_{\ell}					& \mapsto \begin{cases} 0 & \ell \text{ is inert in $E/\Q$} \\ \psi(\frakl) + \psi(\overline{\frakl}) & \ell \text{ splits as $\frakl \overline{\frakl}$ in $E/\Q$} \end{cases}. 
\end{align*}
This map is indeed a homomorphism of $\Lambda$-algebras since after composing with any of the Zariski dense specializations corresponding to characters $\psi: A \to \calO_{K}$ it realizes the map $\mathbf{H}^{\ord}(dM; \calO_{K}) \to \calO_{K}$ coming from the eigenform $f_{\psi}$.
The CM components of $\mathbf{H}^{\ord}(dM; \calO_{K})$ which have $CM$ by $E$ are those which are components of $A$.
The full details of this construction are presented in Theorem 7.1 of \cite{hida_iwasawa_modules}.

\begin{proposition}\label{prop:cm_component_unique}
Suppose that $\mathbf{I}$ is a reduced, irreducible component of $\mathbf{H}^{\ord}(N; \calO_{K})$.
If a CM eigenform of weight $k \geq 2$ arises as a specialization of $\mathbf{I}$, then $\mathbf{I}$ is a CM component, and in particular every specialization of $\mathbf{I}$ has CM by the same imaginary quadratic field.
\end{proposition}
\begin{proof}
Let $f$ be the CM eigenform arising from $\mathbf{I}$.
We know by \cref{thm:hida_etale} that there is a unique component of $\mathbf{H}^{\ord}(N; \calO_{K})$ giving rise to $f$.
By assumption this component is $\mathbf{I}$, however the construction above produces a CM component which specializes to any given CM eigenform.
Thus we must have that $\mathbf{I}$ itself is one of the CM components constructed above.
\end{proof}

We return to the description of Hecke fields, building an explicit description of the Hecke fields of CM components.
This is essentially a combination of \cref{lem:cm_hecke_fields} with the explicit description of CM components given above.

\begin{lemma}\label{lem:cm_family_hecke_fields}
Suppose that $\mathbf{I}$ is a reduced, irreducible component of $\mathbf{H}^{\ord}(N; \calO_{K})$.
Let $E$ be the imaginary quadratic field by which $\mathbf{I}$ has CM, and let $h$ be the class number of $E$.
Fix a weight $k \geq 2$.
There are elements $a_{1}, \ldots, a_{h}$ of $K$ such that any weight $k$ specialization $f$ of $\mathbf{I}$ has
\[
\Q(f) \subseteq E(\mu_{hNp^{\infty}}, a_{1}^{1/h}, \ldots, a_{h}^{1/h}).
\]
In particular the Hecke field of each weight $k$ form arising from $\mathbf{I}$ has its Hecke field contained within a fixed finite extension of the $p$-th cyclotomic field $\Q(\mu_{p^{\infty}})$.
\end{lemma}
\begin{proof}
Suppose that $f_{1}, f_{2}$ are any two (CM) forms of weight $k$ arising from $\mathbf{I}$, each as the theta series attached to an algebraic Hecke character $\psi_{1}, \psi_{2}$.
We know by the construction of $\mathbf{I}$ that the character $\psi_{1}\psi_{2}^{-1}$ has finite $p$-power order (it is an algebraic Hecke character of trivial infinity-type).
Pick $a_{1}, \ldots, a_{h}$ as in \cref{lem:cm_hecke_fields} as applied to the form $f_{1}$.
Since the character $\psi_{1}\psi_{2}^{-1}$ is finite order and moreover has $p$-power order, we see that in the presence of all $p$-power roots of unity (and the required ``tame'' roots of unity of order $hN$) these same $a_{i}$ generate over $E$ a field containing the Hecke field of $f_{2}$.
Since we could take $f_{2}$ to be any form of weight $k$ arising from $\mathbf{I}$, we see that the field $E(\mu_{hNp^{\infty}}, a_{1}^{1/h}, \ldots, a_{h}^{1/h})$ contains the Hecke field of any weight $k$ specialization of $\mathbf{I}$.
\end{proof}

\begin{remark}
Hida's characterization of CM families, stated in this article as \cref{thm:hida_cm_iff_hecke_field}, can be interpreted as a converse of \cref{lem:cm_family_hecke_fields}.
\Cref{lem:cm_family_hecke_fields} shows that the Hecke fields of forms arising from a CM component are uniformly controlled.
Hida's result \cref{thm:hida_cm_iff_hecke_field} shows that any component of $\mathbf{H}^{\ord}(N; \calO_{K})$ which has sufficiently controlled Hecke fields in a single weight must be a CM component.
It is worth noting that \cref{thm:hida_cm_iff_hecke_field} is much weaker than requiring that a component has uniformly controlled Hecke fields; rather it only requires that for a density $1$ set of primes $\ell$, the degree of the ``$\ell$-Hecke field'' $\Q(a_{\ell}(f))$ remains bounded over $\Q(\mu_{p^{\infty}})$ as one varies over forms $f$ of a fixed weight which arise from $\mathbf{I}$.

Under the assumption that the family in question has infinitely many classical forms of low weight, we establish this boundedness of Hecke fields across the entire family using the special properties of Hecke fields in low weight as embodied by \cref{lem:weight_one_hecke_field_bound}, along with the rigidity principles of \cref{sec:rigidity} and construction of \cref{sec:construction} to extend from low weight to regular weight.
\end{remark}
\section{Rigidity principles for $p$-adic power series}\label{sec:rigidity}

In this section we prove rigidity results for integral extensions of $p$-adic power series rings.
These results will be used used to propagate the boundedness of Hecke fields in low weight to regular weight, where the Ramanujan conjecture is known and Hida's theorem (stated as \cref{thm:hida_cm_iff_hecke_field}) relating boundedness of Hecke fields and complex multiplication may be applied.
In particular the boundedness of Hecke fields in low weight is what motivates the conditions of \cref{thm:many_roi_rigidity}; see \cref{sec:big_rep_rigidity} for the application of this theorem to the coefficients of the characteristic polynomial of Frobenius elements in a high-dimensional representation of the absolute Galois group of $\Q$.

We fix the following notation for use in this section.
\begin{itemize}
\item $K$ is a finite extension of $\Q_{p}$, with ring of integers $\calO$, uniformizer $\pi$, and residue field $\F$.
\item $\C_{p}$ is the completion of an algebraic closure of $K$, $\calO_{\C_{p}}$ is the integral closure of $\calO$ within $\C_{p}$, and $\frakm_{\C_{p}}$ is the maximal ideal of $\calO_{\C_{p}}$.
\item As in \cref{sec:notation} $\ord_{p}$ is the valuation on $K$ and extensions thereof, normalized so that $\ord_{p}(p) = 1$.
\item $\Lambda = \calO \llbracket T \rrbracket$ is the ring of formal power series in one variable with coefficients in $\calO$.
\item $M$ is the integral closure of $\Lambda$ in some finite extension of $\Frac(\Lambda)$.
\end{itemize}

\subsection{Weierstrass preparation and Newton polygons}\label{sec:weierstrass_and_newton}

The Weierstrass preparation theorem and the theory of Newton polygons will be the basic tools we use to describe the behaviour of elements of $\Lambda$ and $M$.
Recall that a distinguished polynomial $f(T) \in \calO[T] \subset \Lambda$ is a monic polynomial such that every coefficient other than the leading one is divisible by the uniformizer $\pi$.
With this notion we can state the $p$-adic Weierstrass preparation theorem.

\begin{theorem}[$p$-adic Weierstrass preparation]\label{thm:weierstrass_prep}
If $F(T) \in \Lambda$ is non-zero there is a unique way to write it as
\[
F(T) = \pi^{k}f(T)u(T)
\]
where $k \geq 0$ is an integer, $f(T)$ is a distinguished polynomial, and $u(T)$ is a unit in $\Lambda$ (in other words, the constant term of $u$ is an element of $\calO^{\times}$).
\end{theorem}

For a proof of the $p$-adic Weierstrass preparation theorem, see \cite{lang_cyclotomic_fields}, Chapter 5, Section 2, Theorem 2.2.

\begin{lemma}\label{lem:weierstrass_zero}
If $F(T) \in \Lambda$ is non-zero, then for any $t \in \frakm_{\C_{p}}$ (i.e. $t \in \C_{p}$ and $\ord_{p}(t) > 0$) the series $F(t)$ converges in $\C_{p}$.
Moreover such $F(T)$ have only finitely many roots $t \in \frakm_{\C_{p}}$.
\end{lemma}
\begin{proof}
Let $F(T) = \pi^{k}f(T)u(T)$ in Weierstrass preparation.
If $\ord_{p}(t) > 0$, then $u(t)$ converges since $\ord_{p}(t^{n}) = n\cdot\ord_{p}(t)$ goes to infinity with $n$, and $\ord_{p}(u(t)) = 0$ since the unit constant term of $u(T)$ dominates the norm of any term involving $t$.
Thus $F(t) = \pi^{k}u(t)f(t)$ converges since $f$ is a polynomial and $u$ converges at $t$.
Finally we see that since $u(t)$ is always a unit, we have $F(t) = 0$ if and only if $f(t) = 0$, and $f$ necessarily has finitely many roots in $\C_{p}$ as it is a polynomial.
\end{proof}

\begin{lemma}\label{lem:weierstrass_small_cont}
Let $F(T)$ have Weierstrass preparation $F(T) = \pi^{k}f(T)u(T)$ where $f(T)$ has degree $d$.
If $t \in \frakm_{\C_{p}}$ with $0 < \ord_{p}(t) < \frac{\ord_{p}(\pi)}{d}$, then
\[
\ord_{p}(F(t)) = k \cdot \ord_{p}(\pi) + d \cdot \ord_{p}(t).
\]
\end{lemma}
\begin{proof}
We compute the valuation of $F(t)$ using its Weierstrass preparation
\begin{align*}
\ord_{p}(F(t))	& = \ord_{p}(\pi^{k}) + \ord_{p}(u(t)) + \ord_{p}(f(t)).
\end{align*}
We have that $\ord_{p}(\pi^{k}) = k \cdot \ord_{p}(\pi)$, and $\ord_{p}(u(t)) = 0$ since the unit constant term dominates the norm.
Finally we have that $\ord_{p}(f(t)) = d \cdot \ord_{p}(t)$ since the leading term $t^{d}$ has smaller valuation than any of the other terms of $f(t)$, as $d \cdot \ord_{p}(t) < \ord_{p}(\pi)$ and every other term has valuation at least $\ord_{p}(\pi)$ since $f(T)$ is a distinguished polynomial.
\end{proof}

We recall the construction of Newton polygons for polynomials over $\C_{p}$.
Suppose that $f(X) = \sum_{i=0}^{d} a_{i}X^{i}$ in $\C_{p}[X]$.
We plot the points $(d-i, \ord_{p}(a_{i}))$ in the plane (allowing points ``at $\infty$'' if some of the coefficients $a_{i}$ are $0$ and hence have infinite valuation), and form their lower convex hull.
The resulting set of line segments in the plane is called the \emph{Newton polygon} of $f$.
The usefulness of Newton polygons lies in the fact that this simple combinatorial construction gives us total knowledge of the valuations of the roots of $f$.

% TODO have a Newton polygon figure here

\begin{theorem}\label{thm:newton_polygon}
Suppose that the Newton polygon of $f$ consists of $n$ line segments, with the $i$-th segment having horizontal length $\ell_{i}$ and slope $m_{i}$.
If there is a line segment of infinite slope it must occur at the end, and in that case we consider the length $\ell_{i}$ to be such that $X^{\ell_{i}}$ divides $f$ exactly.
Then for each $i$ in the range $1 \leq i \leq n$ there are $\ell_{i}$ roots of $f$ which have valuation equal to $m_{i}$.
\end{theorem}

See Chapter 3, Section 3 of \cite{koblitz_p-adic} for more information on the theory of Newton polygons.
Note that Koblitz's convention for Newton polygons is slightly different from ours; his Newton polygons are vertical reflections of ours.
Koblitz's convention has the benefit of also applying easily to power series, at the drawback that the slopes of the polygon correspond to inverses of the valuations of the roots.
Our convention is chosen so that the slopes are the valuations, and we won't need to use Newton polygons for non-polynomial power series.

We will use Newton polygons to study specializations of elements of $M$, where $M$ is the integral closure of $\Lambda$ in a finite extension of $\Frac(\Lambda)$.
Suppose that we have a ring homomorphism $P: M \to \C_{p}$ which extends the ring homomorphism $P_{t}: \Lambda \to \C_{p}$ given by $T \mapsto t \in \frakm_{\C_{p}}$.
In a slight abuse of notation we call $P$ a $\C_{p}$-valued point of $M$ (rather than of $\Spec(M)$).
Given $F \in M$, we write $F(P)$ rather than $P(F)$, thinking of $F$ as an ``algebraic'' analytic function, to align with how we think of elements of $\Lambda$ as analytic functions.
Note that if $F \in \Lambda$, $F(P_{t})$ is simply the power series $F$ evaluated at $t$.

\begin{remark}\label{rem:integral_weierstrass_small_cont}
Suppose that we're given $F \in M$, and $P$ is a $\C_{p}$ point of $M$, extending the $\C_{p}$ point $P_{t}$ of $\Lambda$.
If $R(T, X)$ is a monic irreducible polynomial satisfied by $F$, we have that $R(T, F) = 0$ in $M$, and so also $R(t, F(P)) = 0$ in $\C_{p}$.
By computing the Newton polygon of $R(t, X)$ we can obtain the valuation of $F(P)$; in particular for $t \in \frakm_{\C_{p}}$ of sufficiently small valuation we get that each coefficient of $R(t, X)$ has valuation of the form $d_{i} \cdot \ord_{p}(t) + k_{i} \cdot \ord_{p}(\pi)$ as in \cref{lem:weierstrass_small_cont}.
We then have that for $\ord_{p}(t)$ sufficiently small, $\ord_{p}(F(P)) = a \cdot \ord_{p}(t) + b$ for some positive rational $a, b$.
Of course since all of the valuations involved are rational there is always some choice of $a$ and $b$ making the above statement true; the point is that the Newton polygon produces such a choice for us, and those $a$ and $b$ can computed from the Weierstrass preparations of the coefficients of $R(T, X)$.
\end{remark}

\begin{lemma}\label{lem:integral_weierstrass_zero}
Let $F$ be an element of $M$.
Suppose that there is an infinite set $S$ of $\C_{p}$ points of $M$, each extending points $P_{t}$ of $\Lambda$, such that $F(P) = 0$ for all $P \in S$.
Then $F = 0$.
\end{lemma}
\begin{proof}
Suppose that $R(T, X) \in \Lambda[X]$ is a monic irreducible polynomial which $F$ satisfies.
Since $R(t, F(P)) = 0$ for any $P \in S$, the constant term $a_{0}(T)$ of $R(T, X)$ must satisfy $a_{0}(t) = 0$ for each $t$ which one of the points of $S$ lifts.
Since each point $P_{t}$ of $\Lambda$ extends to at most finitely many points of $M$, there must be infinitely many such $t$.
By \cref{lem:weierstrass_zero} since the constant term of $R(T, X)$ has infinitely many roots $t$ in $\frakm_{\C_{p}}$ it must be $0$.
Since $R(T, X)$ is irreducible by assumption, we must have that $R(T, X) = X$, and hence $F = 0$.
\end{proof}

\begin{lemma}\label{lem:polygon_continuity}
Let $R(T, X) \in \Lambda[X]$.
For all $t \in \frakm_{\C_{p}}$ with $\ord_{p}(t)$ sufficiently small, the vertices of the Newton polygon of $R(t, X)$ occur at the same indices.
\end{lemma}
\begin{proof}
The Newton polygon of a monic degree $d$ polynomial is completely determined by the set of valuations of the coefficients.
Thinking about the set of valuations as living in $\R^{d}$, we have a stratification of $\R^{d}$ according to which vertices lie in the Newton polygon.
The condition of an index contributing a vertex to the Newton polygon is given by a collection of linear inequalities; in other words the set of valuations having a given vertex in the Newton polygon is a finite intersection of half-spaces in $\R^{d}$.
The boundary of these half-spaces correspond to multiple vertices lying on the same line segment of the Newton polygon.
Note that even though the valuation map takes values in $\Q \cup \{\infty\}$ rather than $\R$ we are simply working with the defining inequalities over $\R$, and if needs be maybe we may replace any infinite valuations with sufficiently large non-infinite valuations without affecting any of the arguments.

We know that for $t \in \C_{p}$ with $\ord_{p}(t)$ sufficiently small, the coefficients of $R(t, X)$ have valuation of the form $a\cdot\ord_{p}(t) + b$ by \cref{lem:weierstrass_small_cont}.
Say that the $i$-th coefficient of $R(T, X)$ has valuation $a_{i} \cdot \ord_{p}(t) + b_{i}$ for $\ord_{p}(t)$ sufficiently small.
We consider the curve in $\R^{d}$ given by $s \mapsto (a_{1}s + b_{1}, \ldots, a_{d}s + b_{d})$.
Since the image of this curve is an affine line, we have that for a half-space in $\R^{d}$ the curve must satisfy one of the following three possibilities:
\begin{itemize}
\item the curve is contained entirely within the interior of either the half-space or its complement,
\item the curve is contained entirely in the boundary of the half-space,
\item the curve intersects the boundary of the half-space exactly once.
\end{itemize}
Since there are only finitely many affine conditions involved in defining the stratification, we see that the curve will intersect the boundaries of strata transversally only finitely many times.
Therefore for $s \in (0, \epsilon)$ for a sufficiently small $\epsilon$ the image of the curve will be entirely contained within a single stratum (moving from $s = 0$ to $s > 0$ may change strata, but the curve cannot encounter a boundary within a sufficiently small interval above $0$).
Since the valuations of the coefficients of $R(t, X)$ for $\ord_{p}(t)$ sufficiently small land on this curve, we see that the vertices of the Newton polygon occur at the same indices for any $t$ with $0 < \ord_{p}(t) < \epsilon$ and $\ord_{p}(t)$ small enough for each coefficient of $R(T, X)$ to satisfy \cref{lem:weierstrass_small_cont}.
\end{proof}

\subsection{Bounded sums of roots of unity}\label{sec:bounded_sum_rigidity}

In this section we prove our main result on the rigidity of algebraic power series.
By algebraic power series we mean elements of integral extensions $M$ of $\Lambda$.
Our main result (\cref{thm:many_roi_rigidity}) is the following: if an algebraic power series is a sum of at most $B$ roots of unity when specialized at infinitely many points which extend points of the form $T \mapsto \zeta - 1$ for $p$-power roots of unity $\zeta$, then it is a power series which is a linear combination of at most $B$ terms of the form 
\[
(1 + T)^{e} = \sum_{n=0}^{\infty} \binom{e}{n} T^{n}
\]
where for $e \in \Z_{p}$, $\binom{e}{n}$ is the usual binomial coefficient $\binom{e}{n} = \frac{e(e-1) \ldots (e - n + 1)}{n!}$.
We call power series of the form $(1 + T)^{e}$ ``exponential'' power series.
This result is inspired by the rigidity results used by Hida in his work on the relationship between Hecke fields and complex multiplication for ordinary families.
For example see Lemma 5.1 and Proposition 5.2 of \cite{hida_hecke_fields_of_analytic_families_of_modular_forms}, see also \cite{hida_hecke_fields_of_hilbert_modular_analytic_families}, \cite{hida_transcendence}, and \cite{hida_growth_of_hecke_fields} for variations on these statements.
In a different context, similar results are also used in \cite{serban_manin-mumford} and \cite{serban_bianchi} in studying $p$-adic families of automorphic forms over imaginary quadratic fields.

In Hida's work the need for these rigidity lemmas arises in the following way.
Given infinitely many $\ell$-Weil numbers of bounded degree over $\Q(\mu_{p^{\infty}})$ ($\ell$ a prime different from $p$), there are only finitely many up to equivalence (two Weil numbers are equivalent if their quotient is a root of unity, see corollary 2.2. of \cite{hida_transcendence}).
Hence if $F$ is an algebraic power series specializing to $\ell$-Weil numbers at roots of unity, it must be the case that after dividing out by some Weil number we have a power series which takes values in $\mu_{p^{\infty}}$ infinitely often.
The algebraic power series $F$ in question are those interpolating Frobenius eigenvalues across a $p$-ordinary family of modular forms.
Applying the rigidity statement allows us to produce forms in this family with controlled Hecke fields, and from there use those forms to establish that the family has complex multiplication.

We are interested in rigidity statements that apply to algebraic power series specializing to a bounded number of roots of unity infinitely often.
The main difficulty in establishing rigidity statements for algebraic power series specializing to a bounded number of roots of unity, rather than a single root of unity, is that cancellation between different terms can interfere with precise control of valuations.
The following facts about quotients of rings of integers are crucial to putting limits on the possible cancellations that can occur among sums of roots of unity.
\begin{lemma}\label{lem:algebra_quotient}
Suppose that $\calO/\pi^{n}$ has characteristic $p$, i.e. $n$ is less than or equal to the ramification index $e = [\calO:W(\F)]$.
Then $\F[x]/x^{n} \cong \calO/\pi^{n}$, with the isomorphism given by $x \mapsto \pi$.
\end{lemma}
\begin{proof}
Since $\calO/\pi^{n}$ has characteristic $p$, the Teichm\"{u}ller lift map $\F \to \calO/\pi^{n}$ given by $a \mapsto \lim_{m \to \infty} \tilde{a}^{p^{m}}$, where $\tilde{a}$ is any lift of $a$, is an algebra homomorphism.
Consider the map $\F[x] \to \calO/\pi^{n}$ given by $x \mapsto \pi$.
This map is surjective since $\calO/\pi^{n}$ is generated by Teichm\"{u}ller lifts and $\pi$.
The kernel of this map is $(x^{n})$, and so we have the claimed isomorphism.
\end{proof}

\begin{remark}\label{rem:algebra_quotient_cyclotomic}
Suppose that $\zeta$ is a primitive $p^{n}$-th root of unity, and $\calO = W(\F)$ is the ring of integers of an unramified extension of $\Q_{p}$.
Then for $m < n$ we have that $(\zeta^{p^{m}} - 1) = (\zeta - 1)^{p^{m}}$ as ideals of $\calO[\zeta]$ by comparing valuations.
Since $\zeta^{p^{m}} - 1$ has positive valuation less than $1$, the quotient $\calO[\zeta]/(\zeta^{p^{m}} - 1)$ has characteristic $p$, and is a polynomial ring $\F[x]/x^{p^{m}}$ by \cref{lem:algebra_quotient}.
Moreover, by changing variables to $y = x + 1$ we see that
\begin{align*}
\frac{\F[y]}{y^{p^{m}} - 1} 	& \to \frac{\calO[\zeta]}{\zeta^{p^{m}} - 1} \\
y 							& \mapsto \zeta	
\end{align*}
is an isomorphism.
\end{remark}

We begin by proving our main result in the special case of an algebraic power series which takes values in $\mu_{p^{\infty}}$ infinitely often.
This proof serves as a good introduction to the ideas in the proof of \cref{thm:many_roi_rigidity} while being less technical.
The first appearance of this result is in \cite{hida_hecke_fields_of_analytic_families_of_modular_forms}, where two proofs are given; our strategy builds off of the second proof in \cite{hida_hecke_fields_of_analytic_families_of_modular_forms} which Hida credits to Kiran Kedlaya.
% \begin{corollary}[Lemma 5.1 in \cite{hida_hecke_fields_of_analytic_families_of_modular_forms}]\label{cor:one_roi_rigidity}
% Suppose that we are given $F$ satisfying the assumptions of \cref{thm:many_roi_rigidity}, with the additional assumption that $N = 1$ and that the coefficient $a_{1}$ is also a root of unity.
% Then $F$ is of the form $\zeta_{0}(1 + T)^{e}$ for some root of unity $\zeta_{0}$.
% \end{corollary}
% \begin{proof}
% This is \cref{thm:many_roi_rigidity} in the case $N = 1$, noting that the coefficient on $(1 + T)^{e}$ will necessarily be a single root of unity in this case.
% The only possible cancellation pattern in the quotient ring $R_{\zeta}$ is that the exponent $e_{\zeta, 1}$ will necessarily be congruent to $e_{1} = e$ modulo $p^{n-k}$, hence after multiplying by a root of unity in $\mu_{p^{k}}$ we have that $\zeta^{e_{\zeta, 1}} = \zeta_{0}\zeta^{e}$ on an infinite set of $\zeta$.
% Since $F = \zeta_{0}(1 + T)^{e}$ on an infinite set, we have by \cref{lem:integral_weierstrass_zero} that they are equal everywhere.
% \end{proof}

\begin{theorem}[Lemma 5.1 in \cite{hida_hecke_fields_of_analytic_families_of_modular_forms}]\label{thm:one_roi_rigidity}
Suppose that we are given the following data:
\begin{itemize}
\item an element $F$ in an integral extension $M$ of $\Lambda$
\item an infinite set $S \subset \mu_{p^{\infty}}$
\item for each $\zeta \in S$, a $\Qbar_{p}$ point $P_{\zeta}$ of $M$ which extends the point $T \mapsto \zeta - 1$ of $\Lambda$
\end{itemize}
with the property that for each $\zeta \in S$, $F(P_{\zeta})$ is a power of $\zeta$.
Then there is a root of unity $\xi'$ and exponent $e \in \Z_{p}$ such that $F \in \Lambda[\xi']$ and $F = \xi' (1 + T)^{e}$.
\end{theorem}
\begin{remark}
In Hida's formulation of this result (which is stated for power series only rather than elements of integral extensions of $\Lambda$), it is only required that $F(P_{\zeta}) \in \mu_{p^{\infty}}$ for infinitely many $\zeta$.
While this may seem more general in that $F(P_{\zeta})$ could potentially be a $p$-th root of $\zeta$ for all $\zeta$, the control of valuations as in \cref{rem:integral_weierstrass_small_cont} and \cref{lem:polygon_continuity} is enough to show that if we have such an $F \in \Lambda$, then in fact $F(P_{\zeta})$ is a power of $\zeta$ for all $\zeta$ of sufficiently large order.
\end{remark}
\begin{proof}
Suppose that $F(P_{\zeta}) = \zeta^{e_{\zeta}}$ for some integer exponent $e_{\zeta}$.
Since $\Z_{p}$ is compact, the infinite set of $e_{\zeta}$ must have a limit point $e \in \Z_{p}$.
We restrict $S$ to a subset such that $e_{\zeta} \to e$ as the multiplicative order of $\zeta$ goes to $\infty$.
Define $G(T) = (1 + T)^{e}$.
Define $H = F - G$, and let $R(T, X)$ be a monic irreducible polynomial in $\Lambda[X]$ which $H$ satisfies.

On one hand we know from \cref{rem:integral_weierstrass_small_cont} and \cref{lem:polygon_continuity} that there are positive rational numbers $a, b$ such that for $\zeta \in S$ of sufficiently large order we have
\[
\ord_{p}(H(P_{\zeta})) = a \cdot \ord_{p}(\zeta - 1) + b.
\]
On the other hand we can compute directly that if $\zeta \in S$ is of order $p^{n}$, and $e_{\zeta} \equiv e \bmod{p^{m}}$ for $n \geq m$, then
\begin{align*}
\ord_{p}(H(P_{\zeta})) 	& = \ord_{p}(\zeta^{e_{\zeta}} - \zeta^{e}) \\
						& = \ord_{p}(\zeta^{e_{\zeta} - e} - 1) \\
						& \geq \ord_{p}(\zeta^{p^{m}} - 1) \\
						& = p^{m} \cdot \ord_{p}(\zeta - 1).
\end{align*}
since $\zeta^{e_{\zeta} - e}$ has multiplicative order at most $p^{n-m}$. 
Choosing our $\zeta$ of large enough order so that $p^{m} > a$, we see by comparing our two expressions for $\ord_{p}(H(P_{\zeta}))$ that we must have $b > 0$.

Fix a $k$ such that $b > \frac{1}{\varphi(p^{k})}$ (note that $\frac{1}{\varphi(p^{k})} = \ord_{p}(\zeta^{p^{n-k}} - 1)$ for $\zeta$ of order $p^{n}$).
Then if $\zeta \in S$ is a primitive $p^{n}$-th root of unity for $n > k$ and $e_{\zeta}$ is sufficiently $p$-adically close to $e$, we have that $H(P_{\zeta}) = 0$ in the quotient ring $R_{\zeta} = \Z_{p}[\zeta]/(\zeta^{p^{n-k}} - 1)$; this follows by computing valuations since $\ord_{p}(H(P_{\zeta})) \geq \ord_{p}(\zeta^{p^{n-k}} - 1)$ by the above choices.
As in \cref{rem:algebra_quotient_cyclotomic} we have that $R_{\zeta}$ is isomorphic to a truncated polynomial ring $\F_{p}[y]/(y^{p^{n-k}} - 1)$ where the isomorphism sends $y \mapsto \zeta$.
In order for $y^{e_{\zeta}} - y^{e} = 0$ in $R_{\zeta}$ it must be the case that $e_{\zeta} \equiv e \bmod{p^{n-k}}$ for all such $\zeta$.
However there are only finitely many values that $\zeta^{e_{\zeta} - e}$ can take if $e_{\zeta} \equiv e \bmod{p^{n-k}}$ as this must be a $p^{k}$-th root of unity.
So if we choose $\xi'$ such that $\zeta^{e_{\zeta} - e} = \xi'$ for infinitely many $\zeta$, we see that $F - \xi'(1 + T)^{e}$ is $0$ when specialized at infinitely many of the $P_{\zeta}$.
Therefore \cref{lem:integral_weierstrass_zero} shows that $F = \xi'(1 + T)^{e}$.

\end{proof}

We are now in place to prove the main result of this section.
Before doing so we sketch the idea of the proof, which follows the same strategy as \cref{thm:one_roi_rigidity}.
Given an $F$ as in \cref{thm:many_roi_rigidity}, we use the density of the exponents appearing to produce a guess $G(T)$ for the form of $F$ which is a linear combination of exponential power series.
We can show that the difference $F - G$ is $p$-adically close to $0$ under many specializations; the challenge is to show that this is because the terms of $F$ match up with the terms of $G$ to cancel out, rather than the terms of $F$ cancelling out with each other.
By working in the quotient ring by an appropriate power of $(\zeta - 1)$ as in \cref{rem:algebra_quotient_cyclotomic} we are in a polynomial ring, where we can ensure that unexpected cancellations are limited.
Some cancellation between terms may still occur, but we can classify such cancellations into groups of terms which are consistently close to each other $p$-adically.
This grouping allows us to refine our guess $G$, possibly reducing the value $B$, and repeat until we've ruled out all possible unexpected cancellations.

\begin{theorem}\label{thm:many_roi_rigidity}
Suppose that we are given the following data:
\begin{itemize}
\item an element $F$ in an integral extension $M$ of $\Lambda$
\item a constant $B \in \Z_{\geq 0}$
\item an infinite set $S \subset \mu_{p^{\infty}}$
\item for each $\zeta \in S$, a $\Qbar_{p}$ point $P_{\zeta}$ of $M$ which extends the point $T \mapsto \zeta - 1$ of $\Lambda$
\item a root of unity $\xi$
\item coefficients $c_{1}, \ldots, c_{B} \in \Z[\xi]$
\end{itemize}
with the property that for each $\zeta \in S$, $F(P_{\zeta}) \in \Z[\xi, \zeta]$ and $F(P_{\zeta})$ can be written in the form
\begin{align*}
F(P_{\zeta})	& = \sum_{i=1}^{B} c_{i}\zeta^{e_{\zeta, i}}
\end{align*}
for some exponents $e_{\zeta, i}$.
Then there is a root of unity $\xi'$, coefficients $d_{i} \in \Z[\xi']$, and exponents $e_{i} \in \Z_{p}$ such that $F \in \Lambda[\xi']$ and
\begin{align*}
F 	& = \sum_{i=1}^{B} d_{i}(1 + T)^{e_{i}}.
\end{align*}
\end{theorem}
\begin{proof}
The proof proceeds by induction on $B$.
If $B = 0$, we have that $F(P_{\zeta}) = 0$ for infinitely many points $P_{\zeta}$, and \cref{lem:integral_weierstrass_zero} allows us to conclude that $F = 0$, which is of the desired form.
The bulk of the proof is therefore to show that given such an $F$ as in the theorem statement, we may write $F$ in the form $G + F_{1}$, where $G$ is a power series of the desired form (a linear combination of terms of the form $(1 + T)^{e}$) and $F_{1}$ satisfies the assumptions of the theorem with a smaller value of $B$ than that of $F$.

The first step is to construct a candidate expression $G$, and then to show that the specializations of $H = F - G$ at many of the points $P_{\zeta}$ are $p$-adically close to $0$.
Considering the $e_{\zeta, i}$ as integers, we have an infinite set of points in $\Z_{p}^{B}$.
Since $\Z_{p}^{B}$ is compact, the set of tuples $e_{\zeta, i}$ must have a limit point $(e_{1}, \ldots, e_{B}) \in \Z_{p}^{B}$.
Define
\begin{align*}
H 	& = F - \sum_{i=1}^{B} c_{i}(1 + T)^{e}
\end{align*}
as an element of $M[\xi]$, and let $R(T, X) \in \Lambda[X]$ be a monic irreducible polynomial satisfied by $H$.

Let us restrict ourselves to an infinite subset of $S$ such that as the multiplicative order of $\zeta \in S$ goes to infinity we have that $e_{\zeta, i} \to e_{i}$ for each $i$.
We know from \cref{lem:polygon_continuity} that for $\zeta$ of sufficiently large multiplicative order the Newton polygon of $R(\zeta - 1, X)$ is stable, and hence the specialization $H(P_{\zeta})$ for $\zeta \in S$ must have valuation determined by one of the slopes of this polygon.
Passing to a further infinite subset of $S$ we may assume that the specialization has valuation determined by a single line segment in the stable Newton polygon, and hence for $\zeta$ of large multiplicative order we must have that
\[
\ord_{p}(H(P_{\zeta})) = a \cdot\ord_{p}(\zeta - 1) + b
\]
for some fixed rational $a$ and $b$.
Since we know that the $e_{\zeta, i} \to e_{i}$, let us restrict to $\zeta$ of sufficiently large multiplicative order so that $e_{\zeta, i} \equiv e_{i} \bmod{p^{m}}$ for $m$ chosen large enough to ensure that $p^{m} > a$.
Then we have that
\begin{align*}
\ord_{p}(H(P_{\zeta}))	& \geq \min_{i} \left(\ord_{p}(c_{i} (\zeta^{e_{\zeta, i}} - \zeta^{e_{i}}))\right) \\
						& = \min_{i} \left(\ord_{p}(c_{i}\zeta^{e_{i}}(\zeta^{e_{\zeta, i} - e_{i}} - 1))\right) \\
						& \geq \min_{i} \left(\ord_{p}(c_{i})\right) + \ord_{p}(\zeta^{p^{m}} - 1) \\
						& \geq \min_{i} \left(\ord_{p}(c_{i})\right) + p^{m} \cdot\ord_{p}(\zeta - 1)
\end{align*}
using the construction of $H$ and our choice of $\zeta$ large enough ensuring that $e_{\zeta, i} - e_{\zeta}$ is divisible by $p^{m}$.
Combining these two perspectives on $\ord_{p}(H(P_{\zeta}))$ we get that
\begin{align*}
a \cdot\ord_{p}(\zeta - 1) + b & \geq p^{m} \cdot\ord_{p}(\zeta - 1) + \min_{i} \ord_{p}(c_{i}).
\end{align*}
Since $p^{m} > a$ by construction, we see that $b > \min_{i} \ord_{p}(c_{i})$.
In particular let $c$ be one of the coefficients achieving the minimum valuation, then for every $\zeta \in S$ we have that $c^{-1}H(P_{\zeta})$ has valuation at least $v > 0$ for some constant $v$.
Note that $c^{-1}H(P_{\zeta})$ is still an element of $\Z_{p}[\xi, \zeta]$ rather than $\Q_{p}(\xi, \zeta)$ since each coefficient $c_{i}$ has valuation at least that of $c$.

Pick a $k$ such that $v > \frac{1}{\varphi(p^{k})}$.
If $\zeta \in S$ is a primitive $p^{n}$-th root of unity for $n > k$, we know that $\zeta^{p^{n-k}}$ is a primitive $p^{k}$-th root of unity, and by valuations we have that $c^{-1}H(P_{\zeta}) = 0$ in the quotient ring $R_{\zeta} = \Z_{p}[\xi, \zeta]/(\zeta^{p^{n-k}} - 1)$.
As in \cref{rem:algebra_quotient_cyclotomic} we have that this quotient ring $R_{\zeta}$ is isomorphic to a truncated polynomial ring $\F[y]/(y^{p^{n-k}} - 1)$ where $\F$ is the residue field of $\Z_{p}[\xi, \zeta]$ and the isomorphism sends $y$ to $\zeta$.
We restrict $S$ to those $\zeta$ of large enough multiplicative order $p^{n}$ so that
\begin{itemize}
\item if $e_{i} \neq e_{j}$, then $e_{i} \not\equiv e_{j} \bmod{p^{d}}$, and $n - k > d$,
\item if $e_{\zeta, i} \equiv e_{i} \bmod{p^{m}}$ for each $i = 1, \ldots, B$, and $e_{i} \not\equiv e_{j} \bmod{p^{d}}$, then $n-k > m > d$.
In particular this forces $e_{\zeta, i} \not\equiv e_{\zeta, j} \bmod{p^{n-k}}$ if $e_{i} \neq e_{j}$.
\end{itemize}
We know that $c^{-1}H(P_{\zeta}) = 0$ as an element of $R_{\zeta}$; we also have by simply reducing the expression that
\begin{align*}
c^{-1}H(P_{\zeta})	& = c^{-1}\sum_{i=1}^{B}c_{i}(y^{e_{\zeta, i}} - y^{e_{i}})
\end{align*}
in $R_{\zeta}$.
Since this expression is equal to $0$ in $R_{\zeta}$ and $\zeta$ is chosen large enough to ensure that the powers of $y$ associated to $e_{i} \neq e_{j}$ cannot interact in $R_{\zeta}$ (as these powers are distinct mod $p^{n-k}$), there must be cancellation occurring among the terms corresponding to each of the values $e_{i}$.
These cancellations must be some combination of the following three possibilities:
\begin{itemize}
\item the coefficients $c^{-1}c_{i}$ are $0$ in $R_{\zeta}$; this cannot happen for all coefficients, as at least one of these is equal to $1$ since $c = c_{i}$ for some $i = 1, \ldots, B$.
\item a set of the coefficients $c_{i}$ sums to $0$ in $\F$, and the corresponding terms $y^{e_{\zeta, i}}$ have exponents which are congruent ${}\bmod{p^{n-k}}$ (similarly the corresponding terms $y^{e_{i}}$ have exponents congruent ${}\bmod{p^{n-k}}$ which in fact implies they are equal).
\item $e_{\zeta, i} \equiv e_{i} \bmod{p^{n-k}}$.
\end{itemize}

There are finitely many patterns that such cancellations can occur in, so restrict to an infinite set of $S$ such that the same cancellation pattern occurs for each $\zeta$ in the restricted $S$.
If two terms $\zeta^{x}$ and $\zeta^{y}$ have exponents that agree ${}\bmod{p^{n-k}}$, then $\zeta^{y} = \zeta^{x}\zeta_{0}$ for a root of unity $\zeta_{0}$ of order dividing $p^{k}$.
Since there are finitely many such $\zeta_{0}$, we restrict to an infinite subset of $S$ where, after collapsing down terms in the cancellation pattern with exponents congruent ${}\bmod{p^{n-k}}$, the pattern of $p^{k}$-th roots of unity appearing is the same.
Note that this collapsing must occur at least once since not all of the coefficients $c^{-1}c_{i}$ are $0$ in $\F$.

We are now in the situation where for an infinite subset of $S$ we have that
\begin{align*}
F(P_{\zeta}) = \sum_{i=1}^{B} c_{i}' \zeta^{e_{\zeta, i}'}
\end{align*}
where the $c_{i}'$ are in $\Z[\xi, \zeta_{0}]$ for a primitive $p^{k}$-th root of unity $\zeta_{0}$, and for at least some indices $i$ the $e_{\zeta, i}'$ are either equal to $e_{i}$ ${}\bmod{p^{n}}$ for all $\zeta$ or there are several $i$ for which the $e_{\zeta, i}$ are equal for all $\zeta$.
Thus through a combination of combining coefficients with equal $e_{\zeta, i}'$ or subtracting off a term of the form $c_{i}'(1 + T)^{e_{i}}$ we have a new $F_{1} \in M[\xi, \zeta_{0}]$ which satisfies the assumptions of the theorem (with an enlarged integral extension $M$ and base ring $\Z[\xi]$ and) with a smaller value of $B$.
\end{proof}

\subsection{Sums of exponential power series}\label{sec:exponential}

In this section we collect several results on power series of the form $F(T) = \sum_{i=1}^{n} d_{i}(1 + T)^{e_{i}}$.
Given such a power series $F$, define $\pi_{i} = (1 + p)^{e_{i}}$.
The use of $\pi_{i}$ is that the specializations of $F$ that we are interested in (namely at points extending $P_{k, \zeta}: T \mapsto \zeta (1 + p)^{k-1} - 1$) can all be expressed using the $\pi_{i}$:
\[
F(P_{k, \zeta})	= \sum_{i=1}^{n} d_{i}\zeta^{e_{i}} \pi_{i}^{k-1}.
\]
In particular $F(P_{k, \zeta}) \in \Q(\zeta, d_{i}, \pi_{i})$, so if we control the field of definition of the $d_{i}$ and $\pi_{i}$, we have control of the field of definition of $F(P_{k, \zeta})$.

The $F$ that we will use arise from families of modular forms; in particular $F(P_{k, \zeta})$ will be related to Hecke eigenvalues of classical modular forms of weight $k$ and character coming from $\zeta$, and will be algebraic.
Our goal in this section is to show that under assumptions of the algebraicity of $d_{i}$ and $F(P_{k, \zeta})$ we have that the $\pi_{i}$ are algebraic.
We begin with the following combinatorial lemma.

\begin{lemma}\label{lem:vandermonde}
Let $d_{1}, \ldots, d_{n}$ be non-zero algebraic numbers, and let $e_{1}, \ldots, e_{n}$ be distinct $p$-adic integers.
Then there are distinct $p$-power roots of unity $\zeta_{1}, \ldots, \zeta_{n}$ such that the matrix with entries $x_{i, j} = d_{i}\zeta_{j}^{e_{i}}$ has non-zero determinant.
\end{lemma}
\begin{proof}
We proceed by induction on $n$, with the base case $n = 1$ being satisfied by the choice of $\zeta =  1$ since $d_{1} \neq 0$.

Assume by induction that we've chosen $\zeta_{1}, \ldots, \zeta_{m}$ for some $m < n$ such that the $m \times m$ matrix with entries $x_{i, j} = d_{i} \zeta_{j}^{e_{i}}$ for $1 \leq i, j \leq m$ has non-zero determinant.
Choose $k$ large enough so that $e_{1}, \ldots, e_{m+1}$ are distinct modulo $p^{k}$ and such that $\zeta_{1}, \ldots, \zeta_{m}$ are all in $\mu_{p^{k}}$.
Let $\tilde{e}_{i}$ be the unique integer which satisfies $0 \leq \tilde{e}_{i} < p^{k}$ and $e_{i} \equiv \tilde{e}_{i} \bmod{p^{k}}$.
Consider the $(m+1)\times(m+1)$ matrix with entries
\begin{align*}
y_{i, j}	& = \begin{cases} x_{i, j} & j \leq m \\ d_{i}X^{\tilde{e}_{i}} & j = m+1 \\ \end{cases}
\end{align*}
where $X$ is a formal variable.
The determinant of this matrix is thus a polynomial in $X$ which is necessarily non-zero as each matrix entry containing $X$ appears with a different power of $X$ so no cancellation can occur between them, and there is at least one term (the $X^{\tilde{e}_{m+1}}$ term) which appears with a non-zero coefficient, by the inductive hypothesis guaranteeing that the upper left minor of the matrix has non-zero determinant.
The degree of this determinant polynomial is one of the $\tilde{e}_{i}$, hence it is strictly less than $p^{k}$.
Since there are $p^{k}$ roots of unity in $\mu_{p^{k}}$, not every element of $\mu_{p^{k}}$ can be a root of this polynomial, hence there is a choice of $\zeta_{m+1} \in \mu_{p^{k}}$ which produces a non-zero determinant when we set $X = \zeta_{m+1}$.
Note that $\zeta_{m+1} \neq \zeta_{j}$ for any $1 \leq j \leq m$, as that would cause two columns of the matrix to be equal (and the determinant to be $0$).

Finally we conclude that this choice of $\zeta_{m+1}$ satisfies the original claim with $e_{i}$ instead of $\tilde{e}_{i}$: since $e_{i} \equiv \tilde{e}_{i} \bmod{p^{k}}$, we have that $\zeta_{m+1}^{e_{i}} = \zeta_{m+1}^{\tilde{e}_{i}}$ for each $i$.
\end{proof}

This lemma is used in the following proposition to show that the $\pi_{i}$ are algebraic, given algebraicity assumptions on the $d_{i}$ and specializations of $F$.
The strategy of the proof is to realize (powers of) the $\pi_{i}$ as solutions to a system of linear equations in the $d_{i}$ and specializations of $F$.

\begin{proposition}\label{prop:pi_algebraic}
Suppose that $F(T) \in \calO\llbracket T \rrbracket$ is of the form $\sum_{i=1}^{n} d_{i}(1 + T)^{e_{i}}$ for some non-zero $d_{i} \in \calO$ and distinct $e_{i} \in \Z_{p}$.
Assume further that the $d_{i}$ are algebraic, and that there is an integer $k \geq 2$ such that $F(P_{k, \zeta})$ is algebraic for almost all $\zeta \in \mu_{p^{\infty}}$.
Then $\pi_{i} = (1 + p)^{e_{i}}$ is algebraic for each $i$.
\end{proposition}
\begin{proof}
Applying \cref{lem:vandermonde} we see that there exists distinct $p$-power roots of unity $\zeta_{1}, \ldots, \zeta_{n}$ such that the system of equations
\begin{align*}
\begin{array}{ccccccc}
F(P_{k, \zeta_{1}})		& = & d_{1}\zeta_{1}^{e_{1}} \pi_{1}^{k-1} & + & \ldots & + & d_{n}\zeta_{1}^{e_{n}} \pi_{n}^{k-1} \\
\vdots					& = & \vdots &  & \vdots &  & \vdots \\
F(P_{k, \zeta_{n}})		& = & d_{1}\zeta_{n}^{e_{1}} \pi_{1}^{k-1} & + & \ldots & + & d_{n}\zeta_{n}^{e_{n}} \pi_{n}^{k-1} \\
\end{array}
\end{align*}
having ``coefficients'' $d_{i}\zeta_{j}^{e_{i}}$ has a unique solution (the matrix of these coefficients is invertible), the solution being the $\pi_{i}^{k-1}$.
Note that we may choose our $\zeta_{j}$ of sufficiently large multiplicative order to guarantee that the $F(P_{k, \zeta})$ are algebraic.

By Cramer's rule the solutions $\pi_{i}^{k-1}$ to the system of equations above have polynomial expressions in terms of the quantities $F(P_{k, \zeta})$ and $d_{i}\zeta_{j}^{e_{i}}$.
As all of these quantities are algebraic, we conclude that the $\pi_{i}^{k-1}$ are algebraic, and hence the $\pi_{i}$ themselves are algebraic.
\end{proof}

\section{Construction of large Galois representations}\label{sec:construction}

In this section we perform the construction which will allow us to propagate information about the degrees of Hecke fields between different weights in our ordinary families.

The idea behind this construction is to essentially take the trace over a character field of the Galois representations attached to a component of the ordinary Hecke algebra which contains many weight one specializations.
We do this to put ourselves in a situation where the characteristic polynomials of Frobenius will have cyclotomic integer coefficients at many weight one specializations, so that the results of \cref{sec:bounded_sum_rigidity} apply.
This will allow us to propagate the fact that we have bounded Hecke fields in weight one to higher weights, where we are allowing ourselves to utilize the Ramanujan conjecture and we may apply results of Hida to deduce that our component of the ordinary Hecke algebra has complex multiplication.
This link with higher weight will occur in the next section; in this section we content ourselves with performing this trace construction and showing that the results of \cref{sec:bounded_sum_rigidity} apply to the resulting characteristic polynomials of Frobenius.

We set up the following notation for use in this section.
\begin{itemize}
\item Fix a prime $p$.
\item For some large enough finite extension $\calO$ of $\Z_{p}$, $\Lambda = \calO \llbracket T \rrbracket$ is the weight space for $p$-ordinary Hecke algebras.
We use the notation $P_{k, \zeta}$ for the map $\Lambda \to \Qbar_{p}$ given by $T \mapsto \zeta(1 + p)^{k-1} - 1$; where it will not cause confusion we also use $P_{k, \zeta}$ as notation for the kernel of this map.
\item Fix a tame level $N \nmid p$.
\item We let $\mathbf{H}^{\ord}$ be the $\Lambda$-adic ordinary Hecke algebra $\mathbf{H}^{\ord}(N; \calO)$ with tame level $N$.
\end{itemize}

\subsection{Selecting components}\label{sec:selecting_components}

Our basic assumption will be that we have a component $\mathbf{I}$ of $\mathbf{H}^{\ord}$ which specializes to infinitely many classical weight one eigenforms.
By an extended pigeonhole principle argument, we select several more components of $\mathbf{H}^{\ord}$ with the property that together these components see all Galois conjugates of the classical weight one forms arising from $\mathbf{I}$.

\begin{theorem}\label{thm:components}
Suppose that $\mathbf{I} = \mathbf{H}^{\ord}/\frakP$ is a reduced, irreducible component of $\mathbf{H}^{\ord}$, with the property that there are infinitely many classical weight one eigenforms arising as specializations of $I$.
Then there is an infinite set $R$ of classical weight one eigenforms arising from $\mathbf{I}$ such that the following hold.
\begin{enumerate}
\item[(1)] There is an integer $m \leq \rank_{\Lambda}(\mathbf{H}^{\ord})$ such that each $f \in R$ has exactly $m$ Galois conjugates over its character field.
\item[(2)] There are reduced, irreducible components $\mathbf{I}_{i} = \mathbf{H}^{\ord}/\frakP_{i}$ of $\mathbf{H}^{\ord}$ for $i = 1, \ldots, m$ and for each $f \in R$ there is a $\Qbar_{p}$-point $P_{f, i}$ of $\mathbf{I}_{i}$ with the following property.
In some ordering $f_{1}, \ldots, f_{m}$ of the $m$ Galois conjugates of $f$ over its character field, the system of eigenvalues of $f_{i}$ arises as the specialization of $\mathbf{I}_{i}$ at the point $P_{f, i}$.
We may take $f_{1} = f$ and $\mathbf{I}_{1} = \mathbf{I}$.
\item[(3)] There exists a finite Galois extension $\Frac(M)$ of $\Frac(\Lambda)$ with $M$ the integral closure of $\Lambda$, together with fixed embeddings $e_{i}: \mathbf{I}_{i} \to M$.
\item[(4)] For each $f \in F$ there is a $\Qbar_{p}$-point $P_{f}$ of $M$ with the property that for each $i = 1, \ldots, m$, we have that
\[
P_{f}|_{e_{i}(\mathbf{I}_{i})} = P_{f, i}.
\]
\item[(5)] Define $r: G_{\Q} \to \oplus_{i=1}^{m} \GL_{2}(\Frac(\mathbf{I}_{i}))$ to be the direct sum of the Galois representations attached to each of the $\mathbf{I}_{i}$.
Using the embeddings $e_{i}$ we can think of $r$ as having image in $\GL_{2m}(\Frac(M))$.
Denote $M_{P_{f}}$ the localization of $M$ at the kernel of $P_{f}$; then for each $f \in R$, the image of $r$ lands in $\GL_{2m}(M_{P_{f}})$, and so can be pushed forward through $P_{f}$ to obtain a representation $r_{f}: G_{\Q} \to \GL_{2m}(\Qbar_{p})$.
\end{enumerate}
\end{theorem}
\begin{proof}
The theorem is an extended application of the pigeonhole principle.
We show the conclusions of the theorem one at a time; at each stage we refine the results from the previous step while maintaining an infinite set of weight one forms.

We know from \cref{lem:weight_one_hecke_field_bound} that any weight one form which is parameterized by $\mathbf{H}^{\ord}$ has at most $\rank_{\Lambda}(\mathbf{H}^{\ord})$ Galois conjugates over its character field.
Thus from our initial set of infinitely many classical weight one forms there must be an $m \leq \rank_{\Lambda}(\mathbf{H}^{\ord})$ which occurs infinitely often as the number of Galois conjugates over the character field.
Restrict to only those $f$ having exactly $m$ Galois conjugates over their character field, and call this set $R_{1}$.
This shows (1).

To show (2), we use that $\mathbf{H}^{\ord}$ has only finitely many irreducible components since it is finite over $\Lambda$, which is irreducible.
For a given $f \in R_{1}$, its $m$ Galois conjugates over its character field arise from some set of $m$ components of $\mathbf{H}^{\ord}$.
Since there are only finitely many possible such sets of components, one must occur for infinitely many $f \in R_{1}$.
Let $R_{2}$ be an infinite subset of $R_{1}$ for which all $f \in R_{2}$ have their Galois conjugates arising from the same set of $m$ components.
Pick this set of components and label them $\mathbf{I}_{1}, \ldots, \mathbf{I}_{m}$ such that (in some ordering) the conjugates $f_{1}, \ldots, f_{m}$ satisfy that $f_{i}$ arises from $\mathbf{I}_{i}$.
Without loss of generality we may assume that $\mathbf{I} = \mathbf{I}_{1}$ and $f = f_{1}$.
Note that for each $f \in R_{2}$ we have a $\Qbar_{p}$-point $P_{f, i}$ of $\mathbf{I}_{i}$ with the property that the specialization of $\mathbf{I}_{i}$ at $P_{f, i}$ is the system of Hecke eigenvalues of $f_{i}$.

The integral extension $M/\Lambda$ as in (3) can be constructed as follows.
Each $\mathbf{I}_{i}$ is an integral domain finite over $\Lambda$, so $\Frac(\mathbf{I}_{i})$ is a finite extension of $\Frac(\Lambda)$.
Take the Galois closure of the compositum of these fields $\Frac(\mathbf{I}_{i})$; this is some finite extension of $\Frac(\Lambda)$, and we take $M$ as the integral closure of $\Lambda$ inside that field.
If $P$ is the point of $\mathbf{I}$ giving rise to $f$, choose any extension of $P$ to a point $P_{f}$ of $M$.

For a given $f \in R_{2}$, we have points $P_{f, i}$ of $\mathbf{I}_{i}$ for $i = 1, \ldots, m$.
Since the Galois action is transitive on points in fibers of $M/\Lambda$, there is some embedding of $\mathbf{I}_{i}$ into $M$ such that $P_{f, i}$ is the restriction of $P_{f}$ to the image of $\mathbf{I}_{i}$.
As there are only finitely many embeddings $\mathbf{I}_{i} \to M$ for each $i$, there are only finitely many possible choices total.
Since for each $f \in R_{2}$ there is at least one choice of embeddings $\mathbf{I}_{i} \to M$ with the desired compatibility between points, and $R_{2}$ is infinite, there must be some choice of embeddings $e_{i}: \mathbf{I}_{i} \to M$ such that for an infinite subset $R_{3} \subseteq R_{2}$ we have the desired compatibility of points.
This shows (3) and (4) of the theorem.

As in the statement of the theorem we define $r$ to be the representation 
\[
r: G_{\Q} \to \oplus_{i=1}^{m} \GL_{2}(\Frac(\mathbf{I}_{i})) \overset{\oplus_{i=1}^{m} e_{i}}{\hookrightarrow} \GL_{2m}(\Frac(M))
\]
obtained as the direct sum of the Galois representations attached to each $\mathbf{I}_{i}$, viewed as having coefficients in $\Frac(M)$.
We would like to be able to specialize $r: G_{\Q} \to \GL_{2m}(\Frac(M))$ through the map $P_{f}: M \to \Qbar_{p}$ in order to recover the representations attached to each $f_{i}$; however without knowing that $M$ is a unique factorization domain it may not be possible to find a basis of $V = \Frac(M)^{2m}$ in which $r$ takes values in $\GL_{2m}(M)$.
For a given $P_{f}$ we can always extend the point $P_{f}$ to have domain $M_{P_{f}}$ (localization of $M$ at the kernel of $P_{f}$), so it will suffice to show that $r$ has coordinates in $M_{P_{f}}$ for an infinite subset of $R_{3}$.

The group $G_{\Q}$ is compact, so there is a lattice $\mathcal{L} \subset \Frac(M)^{2m}$ which is stable under the action of $G_{\Q}$ through $r$.
Each element of $r$ can thus be though of equally well as an element of $\End_{M}(\mathcal{L})$.
We claim that $\End_{M}(\mathcal{L})$ is a finitely generated $M$-module.
Let $a_{1}, \ldots, a_{n}$ be a set of generators for $\mathcal{L}$ as an $M$-module.
Given an $n \times n$ matrix $X = \{x_{i, j}\}$ with coefficients in $M$ we say that $X$ descends to $\mathcal{L}$ if the map given by $a_{i} \mapsto \sum_{j=1}^{n} x_{i, j} a_{j}$ defines an endomorphism of $\mathcal{L}$; note that for any endomorphism of $\mathcal{L}$ there is at least one such matrix.
The subset of $\End_{M}(M^{n})$ of those matrices which descend to $\mathcal{L}$ is an $M$-submodule.
Since $M$ is noetherian (it is a finite extension of the noetherian ring $\Lambda$) we have that any submodule of the finitely generated $\End_{M}(M^{n})$ is finitely generated.
In particular one such submodule surjects onto $\End_{M}(\mathcal{L})$, proving that it is finitely generated.

Viewing each generator of $\End_{M}(\mathcal{L})$ as an element of $\GL_{2m}(\Frac(M))$, we see that it has at most finitely many ``poles'', where by ``pole'' we mean a $\Qbar_{p}$ point $P$ of $M$ such that the entries of that element of $\GL_{2m}(\Frac(M))$ are not in $M_{P}$.
Since there are finitely many generators of $\End_{M}(\mathcal{L})$, each with finitely many poles, we see that there are at most finitely many $\Qbar_{p}$ points of $M$ which can arise as poles of an element $r(g)$.
After removing finitely many of the points in $R_{3}$ in order to avoid these poles, we obtain a set $R$ and have that the image of $r$ lands in $\GL_{2m}(M_{P_{f}})$ for each $f \in R$.
\end{proof}

The representation $r$ defined in \cref{thm:components} is what we'll use to propagate control of Hecke fields from weight one into regular weight.
Since our main focus will be on the specializations of $r$ through the primes $P_{f}$ of $M$, we set $r_{f}$ to be that specialization
\[
r_{f}: G_{\Q} \to \GL_{2m}(\Frac(M)) \overset{P_{f}}{\to} \GL_{2m}(\Qbar_{p})
\]
which is well-defined by part (5) of \cref{thm:components}.
We begin with some basic properties of the representations $r$ and $r_{f}$.
\begin{lemma}\label{lem:big_rep_properties}
The representation $r: G_{\Q} \to GL_{2m}(\Frac(M))$ satisfies the following properties.
\begin{enumerate}
\item[(1)] For primes $\ell \nmid  Np$, $r$ is unramified at $\ell$.
\item[(2)] For every $g \in G_{\Q}$, the characteristic polynomial of $r(g) \in \GL_{2m}(\Frac(M))$ has coefficients in $M$.
\item[(3)] For every $f \in R$, the representation $r_{f}: G_{\Q} \to \GL_{2m}(\Qbar_{p})$ is equal to the direct sum of the $p$-adic Galois representations attached to the conjugates $f_{1}, \ldots, f_{m}$ of $f$.
\end{enumerate}
\end{lemma}
\begin{proof}
(1) is immediate as $r$ is constructed as a direct sum of representations which are unramified at primes $\ell \nmid Np$.
Part (2) is a consequence of continuity of the representation as we now show.
The characteristic polynomial map $G_{\Q} \to \Frac(M)[X]$ given by $g \mapsto \det(XI - r(g))$ is continuous since $r$ itself is continuous and taking characteristic polynomials is continuous.
For primes $\ell$ as above, each of the direct summands of $r$ has the property that characteristic polynomials of $\Frob_{\ell}$ land in $\mathbf{I}_{i}$; the trace and determinant are Hecke operators and so are in $\mathbf{I}_{i}$ rather than $\Frac(\mathbf{I}_{i})$.
Since the characteristic polynomial of the direct sum is simply the product of the characteristic polynomials, we see that the characteristic polynomial of $r(\Frob_{\ell})$ has coefficients in $M$.
Finally since the $\Frob_{\ell}$ are topologically dense in $G_{\Q}$ by the Cebotarev density theorem, we see that every element in the image of $r$ must have characteristic polynomial in $M[X]$ since it is a closed subset of $\Frac(M)[X]$.

Part (3) of this lemma is a consequence of parts (3) and (4) of \cref{thm:components}.
By our choice of embeddings we have that $P_{f}$ restricted to $\mathbf{I}_{i}$ produces the system of Hecke eigenvalues of $f_{i}$; hence $r: G_{\Q} \to \oplus_{i=1}^{2m} \GL_{2}(\Frac(\mathbf{I}_{i}))$ will specialize to the direct sum of the $\rho_{f_{i}, p}$.
\end{proof}

Since the representation $r$ is unramified at primes $\ell \nmid Np$, we introduce notation for the characteristic polynomial of $r(\Frob_{\ell})$.
Let
\[
A_{\ell}(X) = \sum_{j=0}^{2m} A_{\ell, j}X^{j} = \det(X I - r(\Frob_{\ell})).
\]
Note that the coefficients $A_{\ell, j}$ of $A_{\ell}(X)$ lie in $M$ as per \cref{lem:big_rep_properties}.
The key property of the representation $r$ is contained in the following theorem, where we show how the choices in the construction of $r$ lead to control over the field of definition of the specializations $r_{f}$. 
In particular studying the specializations $r_{f}$ will amount to studying the specialized characteristic polynomials of Frobenius $A_{\ell}(P_{f})(X) = \det(XI - r_{f}(\Frob_{\ell}))$.

\begin{theorem}\label{thm:big_rep_coefficient_field}
Suppose that $f \in R$ and $P_{f}$ is the corresponding $\Qbar_{p}$-point of $M$ as in \cref{thm:components}.
Then for a prime $\ell \nmid Np$ we have that the characteristic polynomial $A_{\ell}(P_{f})(X)$ of $r_{f}(\Frob_{\ell})$ has coefficients in the character field of $f$.
\end{theorem}
\begin{proof}
Let $f_{1}, f_{2}, \ldots, f_{m}$ be the $m$ Galois conjugates of $f = f_{1}$.
Let $\Q(\epsilon)$ and $\Q(f)$ be, respectively, the character and Hecke fields of $f$.
For each $f_{i}$, let $\alpha_{i}$, $\beta_{i}$ be the eigenvalues of $\rho_{f_{i}, p}(\Frob_{\ell})$ (so that the $T_{\ell}$-eigenvalue of $f_{i}$ is equal to $\alpha_{i} + \beta_{i}$).
We then have by \cref{lem:big_rep_properties} that $A_{\ell}(P_{f})(X)$ factors as
\begin{align*}
A_{\ell}(P_{f})(X)		& = \prod_{i=1}^{m} (X - \alpha_{i})(X - \beta_{i})
\end{align*}
since $r_{f}$ is the direct sum of the representations $\rho_{f_{i}, p}$.
The coefficients $A_{\ell, j}(P_{f})$ of $A_{\ell}(P_{f})(X)$ are thus (up to sign) the elementary symmetric polynomials of degree $2m$ evaluated at the $\alpha_{i}$ and $\beta_{i}$.

We know that $\alpha, \beta$ are in a degree at most $2$ extension of $\Q(f)$ since they satisfy a degree 2 polynomial with coefficients in $\Q(f)$ (the characteristic polynomial of $\rho_{f, p}(\Frob_{\ell})$).
For any positive integer $n$, the expression $\alpha^{n} + \beta^{n}$ is invariant under switching these two roots, hence it must itself be in $\Q(f)$.
Therefore we have that the power sum $\sum_{i=1}^{m} \alpha_{i}^{n} + \beta_{i}^{n}$ is in $\Q(\epsilon)$, since it is a field trace:
\begin{align*}
\sum_{i=1}^{m} \alpha_{i}^{n} + \beta_{i}^{n}	& = \Trace_{\Q(\epsilon)}^{\Q(f)}(\alpha^{n} + \beta^{n}).
\end{align*}

Since both the power sums and the elementary symmetric polynomials are generating sets (over $\Q$ or extensions thereof) for the space of symmetric polynomials, all of the elementary symmetric polynomials can be expressed as polynomial combinations with rational coefficients of the power sums.
We've shown that the power sums of the $\alpha_{i}, \beta_{i}$ all lie in the character field $\Q(\epsilon)$, so the coefficients of $A_{\ell}(P_{f})(X)$ will also lie in $\Q(\epsilon)$.
\end{proof}

\subsection{Rigidity of Frobenius characteristic polynomials}\label{sec:big_rep_rigidity}

The goal is this section is to show that the coefficients of the characteristic polynomial of $r(\Frob_{\ell})$ are constrained as in \cref{thm:many_roi_rigidity}.
We've seen in the previous section that the specialized characteristic polynomials take values in cyclotomic fields (the characteristic polynomial of $r_{f}(\Frob_{\ell})$ has coefficients in the character field of $f$).
Since these characteristic polynomials are also necessarily integral, the coefficients must be \emph{cyclotomic integers}, that is, elements of $\Z[\zeta]$ for some root of unity $\zeta$.
Following Cassels \cite{cassels} and Loxton \cite{loxton} we define the following quantities for a cyclotomic integer $\alpha$:
% TODO fix the house macro to not be stupid
\begin{itemize}
\item $N(\alpha) = $ the minimal natural number $n$ such that $\alpha$ can be written as a sum of $n$ roots of unity.
\item $\house{\alpha} = $ the maximum of the complex absolute values of $\alpha$.
This is called the \emph{house} of $\alpha$.
\end{itemize}
The following theorem of Loxton relates these two measures of the size of $\alpha$; this is the key link between the bounds we know on the sizes of Hecke eigenvalues and the rigidity results of \cref{sec:bounded_sum_rigidity}.

\begin{theorem}[Loxton, Theorem 1 of \cite{loxton}]\label{thm:loxton}
Choose a real number $d$ with $d > \log{(2)}$.
Then there is a positive constant $c$ depending only on $d$ such that if $\alpha$ is a cyclotomic integer with 
\[
N(\alpha) = n
\]
then 
\[
\house{\alpha}^{2} \geq c\cdot n \cdot\exp(-d \log{(n)}/\log{(\log{(n)})}).
\]
\end{theorem}

As explained in the introduction of \cite{loxton}, this theorem allows us to bound $N(\alpha)$ if we have a bound on $\house{\alpha}$.
In particular if we know that
\[
\house{\alpha}^{2} < c \cdot n \cdot \exp(-d \log{(n)}/\log{(\log{(n)})})
\]
then it must be the case that $N(\alpha) < n$.
Since the expression $c \cdot n \cdot \exp(-d \log{(n)}/\log{(\log{(n)})})$ is increasing in $n$, if we have an absolute bound $\house{\alpha}^{2}$ for some collection of cyclotomic integers $\alpha$, then it forces an absolute bound on the $N(\alpha)$.

\begin{lemma}\label{lem:bounded_house}
There is a constant $C_{\ell, j}$, depending only on $\ell$ and $j$, such that for each $f \in R$ we have
\[
\house{A_{\ell, j}(P_{f})} \leq C_{\ell, j}.
\]
\end{lemma}
\begin{proof}
This follows from the fact that the Frobenius eigenvalues $\alpha$ and $\beta$ of $\rho_{f, p}(\Frob_{\ell})$ have house bounded by a polynomial in $\ell$ by \cref{thm:trivial_coefficient_bound}.
The construction in the proof of \cref{thm:big_rep_coefficient_field} shows that $A_{\ell, j}(P_{f})$ is a polynomial expression with rational coefficients in $\alpha$ and $\beta$ and their Galois conjugates, which all satisfy the same Archimedean bound coming from \cref{thm:trivial_coefficient_bound}.
Applying the triangle inequality liberally to the expression for $A_{\ell, j}(P_{f})$ we obtain a bound on $\house{A_{\ell, j}(P_{f})}$ which is polynomial in the bound on $\house{\alpha}, \house{\beta}$.
Since this polynomial expression in $\alpha$ and $\beta$ and their conjugates is the same for all $f$ in $R$ , and the bound on $\house{\alpha}, \house{\beta}$ is the same for all $f$ in $R$, we obtain a uniform (in $f$) upper bound on $\house{A_{\ell, j}(P_{f})}$.
\end{proof}

\begin{lemma}\label{lem:coefficient_rigidity}
Fix a prime $\ell$ of $F$ such that $\ell \nmid N p$.
Then for each $j$ in the range $0 \leq j \leq 2m$ the coefficient $A_{\ell, j}$ of the characteristic polynomial of $r(\Frob_{\ell})$ satisfies the assumptions of \cref{thm:many_roi_rigidity}.
\end{lemma}
\begin{proof}
For any $f \in R$, we have by \cref{thm:big_rep_coefficient_field} that $A_{\ell, j}(P_{f})$ is in the character field $\Q(\epsilon)$ which is a cyclotomic field.
Since $A_{\ell, j}(P_{f})$ is an elementary symmetric polynomial evaluated at integral inputs (the eigenvalues of $\rho_{f, p}(\Frob_{\ell})$ are integral), it is integral, and hence is a cyclotomic integer.

Fix $j$, and let $C_{\ell, j}$ be the upper bound on all $\house{A_{\ell, j}(P_{f})}$ established in \cref{lem:bounded_house}.
Choose $n$ sufficiently large so that
\[
C_{\ell, j}^{2} < c \cdot n \cdot \exp(-\log{(n)}/\log{(\log{(n)})})
\]
where $c$ is the constant associated to $d = 1 > \log{(2)}$ of \cref{thm:loxton}.
Choosing such an $n$ is possible since the function on the right is unbounded in $n$.
\Cref{thm:loxton} then guarantees for us that $N(A_{\ell, j}(P_{f})) < n $, i.e. each $A_{\ell, j}(P_{f})$ can be written as a sum of fewer than $n$ roots of unity.

Let $\xi$ be a root of unity that generates the tame (i.e. prime to $p$) part of the character fields $\Q(\epsilon)$ (changing $f$ only changes the $p$-power roots of unity present).
Among the finitely many combinations of $B < n$ and possible powers of $\xi$ used to write an integral element of $\Q(\epsilon)$ as a sum of $B$ roots of unity, we pick one that occurs infinitely often among the $A_{\ell, j}(P_{f})$ for $f \in R$.
For the subset $R'$ of $R$ where this combination occurs, we let $B$ and $c_{1}, \ldots, c_{B}$ be the chosen values, and
\[
S = \{\zeta \in \mu_{p^{\infty}}: \text{ there is $f \in R'$ such that $P_{f}$ extends the point $P_{1, \zeta}$ of $\Lambda$}\}.
\]
With this data $A_{\ell, j}$ satisfies the assumptions of \cref{thm:many_roi_rigidity} and hence is a linear combination of exponential power series.
\end{proof}
\section{Bounded Hecke fields}\label{sec:bounded}

Now that we have constructed our representation $r$ and controlled the form of its characteristic polynomials of Frobenius using the rigidity results of \cref{sec:rigidity}, we are in a good position to specialize in regular weight.
The advantage of regular weight is that we can apply results of Hida on the complexity of Hecke fields associated to non-CM ordinary families in order to conclude that our family has CM.
The flavour of Hida's results is that if the Hecke fields of the forms in an ordinary family are sufficiently complicated (as measured by their degree relative to the $p$-cyclotomic extension $\Q(\mu_{p^{\infty}})$), then the family cannot have complex multiplication.
Put another way, if  the Hecke fields of an ordinary family are sufficiently bounded then that family has CM.

Hida has published several variations on theorems of this flavour.
In \cref{sec:hida_bounded} we sketch the proof of the version of this result that we use.
We then assemble the results of \cref{sec:construction} together with Hida's Hecke field result to prove our main theorem in \cref{sec:proof}.

We keep the notation introduced at the start of \cref{sec:construction}.
Most importantly $\mathbf{H}^{\ord}$ is a $\Lambda$-adic Hecke algebra parametrizing $p$-ordinary modular forms having a fixed tame level and character.

\subsection{Hida's results on bounded Hecke fields}\label{sec:hida_bounded}

Hida has proven several variations on theorems of this flavour (see \cite{hida_hecke_fields_of_analytic_families_of_modular_forms}, \cite{hida_hecke_fields_of_hilbert_modular_analytic_families}, \cite{hida_transcendence}, \cite{hida_growth_of_hecke_fields}).
In this section we build off of \cite{hida_transcendence} as it works with Hecke fields away from $p$ (i.e. degrees of $a_{\ell}$ for $\ell\nmid N p$ rather than $a_{p}$).
We found it to be more convenient to take this approach rather an approach that relies on bounding $a_{p}$, although in principle such a strategy could also work in our situation.
We offer a sketch of the proof of this theorem; the main idea is to use \cref{thm:one_roi_rigidity} to find two eigenforms $f$ and $g$ whose $p$-adic Galois representations are ``too similar'' unless $\mathbf{I}$ is a CM family.

Of course we could also attempt to use this result directly for elliptic modular forms of weight one, as the Ramanujan conjecture is known for these forms!
However, our goal is to access our main result without utilizing the Ramanujan conjecture in low weight, so that there is the possibility of our strategy generalizing to the case of Hilbert modular forms of partial weight one.

\begin{theorem}[Hida, Theorem 3.1 of \cite{hida_transcendence}]\label{thm:hida_cm_iff_hecke_field}
Suppose that we are given the following data:
\begin{enumerate}
\item a set $\Sigma$ of primes of $\Q$ of positive density
\item for each $\ell \in \Sigma$ a constant $C_{\ell}>0$
\item an infinite set of $p$-power roots of unity $S$
\item a fixed integer $k \geq 2$
\item a reduced, irreducible component $\mathbf{I} = \mathbf{H}^{\ord}/\frakP$ of $\mathbf{H}^{\ord}$
\item for each $\zeta \in S$, a point $P_{k, \zeta}$ of $\mathbf{I}$ extending the point $P_{k, \zeta}$ of $\Lambda$
\end{enumerate}
with the property that for each $\zeta \in S$, the specialization $\mathbf{I}(P_{k, \zeta})$ (which is a classical modular eigenform $f_{\zeta}$) has its Hecke fields satisfying the following bounds
\[
[\Q(\mu_{p^{\infty}}, a_{\ell}(f_{\zeta})): \Q(\mu_{p^{\infty}})] \leq C_{\ell}
\]
for each $\ell \in \Sigma$.
Then $\mathbf{I}$ has complex multiplication.
\end{theorem}
\begin{proof}[Proof Sketch.]
For $\ell \in \Sigma$ let $\alpha_{\ell}$ be a choice of root of the characteristic polynomial of $\rho_{\mathbf{I}}(\Frob_{\ell})$ where $\rho_{\mathbf{I}}$ is the $p$-adic Galois representation attached to $\mathbf{I}$.
We assume that $\mathbf{I}$ is large enough to contain $\alpha_{\ell}$, extending it and the points $P_{k, \zeta}$ if necessary.
Since $k \geq 2$ the Ramanujan conjecture is known for the specializations $f_{\zeta}$ of $\mathbf{I}$ under $P_{k, \zeta}$.
As a consequence of the Ramanujan conjecture we have that $\alpha_{\ell}(P_{k, \zeta})$ is an $\ell$-Weil number.
Our condition bounding the degrees of Hecke fields yields that $\alpha_{\ell}(P_{k, \zeta})$ has degree at most $2C_{\ell}$ over $\Q(\mu_{p^{\infty}})$.
As there are only finitely many such $\ell$-Weil numbers up to equivalence (see corollary 2.2 of \cite{hida_transcendence}), we pick $\pi_{\ell}$ which occurs infinitely often up to equivalence as the specialization $\alpha_{\ell}(P_{k, \zeta})$ for $\zeta \in S$.
Thus we have that $\pi_{\ell}^{-1} \alpha_{\ell}$ specializes to a $p$-power root of unity infinitely often, hence by a generalization of \cref{thm:one_roi_rigidity} which allows for $p$-power roots of $(1 + T)$(see Proposition 4.1 of \cite{hida_transcendence}), $\alpha_{\ell}$ is of the form $\pi_{l}(1 + T)^{e_{\ell}}$ with $e_{\ell} \neq 0 \in \Q_{p}$ for each $\ell \in \Sigma$.

Choose a $p$-power root of unity $\zeta \neq 1$, and consider the forms $f$ and $g$ which arise as specialization of $\mathbf{I}$ by $P_{k, 1}$ and $P_{k, \zeta}$.
We assume for a contradiction that neither $f$ nor $g$ has complex multiplication.

We let $\alpha_{\ell}(f) = \alpha_{\ell}(P_{k, 1})$ be the Frobenius eigenvalue of $f$ produced by $\alpha_{\ell}$, and similarly $\alpha_{\ell}(g) = \alpha_{\ell}(P_{k, \zeta})$.
By our control of $\alpha_{\ell}$ we know that $\alpha_{\ell}(g) = \zeta^{e_{\ell}} \alpha_{\ell}(f)$.
Since the characteristic polynomial of $\rho_{f}(\Frob_{\ell})$ (with any choice of coefficients) has constant term a power of $\ell$ times a root of unity, we see that a similar relationship holds with the second eigenvalue of Frobenius of each form $f, g$.
Choosing a prime $q$ which splits completely in $\Q(f, g)$ (the compositum of the Hecke fields $\Q(f)$ and $\Q(g)$) for convenience, we see that if $\zeta^{m} = 1$ that
\[
\Trace(\rho_{f, q}(\Frob_{\ell})^{m}) = \Trace(\rho_{g, q}(\Frob_{\ell})^{m})
\]
for each $\ell \in \Sigma$.
Using that $\Trace(\rho^{m}) = \Trace(\Sym^{m}(\rho)) - \Trace(\Sym^{m-2}(\rho) \otimes \det(\rho))$ for any $2$-dimensional representation $\rho$, we obtain
\begin{align*}
& \Trace\left(\Sym^{m}(\rho_{f, q}) \oplus \left(\Sym^{m-2}(\rho_{g, q}) \otimes \det(\rho_{g, q})\right)\right) \\
= & \Trace\left(\Sym^{m}(\rho_{g, q}) \oplus \left(\Sym^{m-2}(\rho_{f, q}) \otimes \det(\rho_{f, q})\right)\right)
\end{align*}
when evaluated on any $\Frob_{\ell}$ for $\ell \in \Sigma$.

Since $f$ and $g$ are not CM forms by assumption, we claim that for $q$ sufficiently large the image of their $q$-adic Galois representations contains an open subgroup of $\SL_{2}(\Z_{q})$.
The residual representations contain $\SL_{2}(\F_{q})$ for large enough $q$ (see Section 0.1 of \cite{dimitrov}), and the classification of compact subgroups of $\SL_{2}(\Z_{q})$ shows that we must therefore have an open subgroup of $\SL_{2}(\Z_{q})$ in the image.
In particular since the representations in question are irreducible and $\Sigma$ has positive density, a result of Rajan (see Theorem 2 of \cite{rajan}) guarantees that we have an equality of representations
\[
\Sym^{m}(\rho_{f, q}) \oplus \left(\Sym^{m-2}(\rho_{g, q}) \otimes \det(\rho_{g, q})\right) = \Sym^{m}(\rho_{g, q}) \oplus \left(\Sym^{m-2}(\rho_{f, q}) \otimes \det(\rho_{f, q})\right)
\]
when restricted to a finite index subgroup $G_{K}$ of $G_{\Q}$.
We also have that $\Sym^{m}(\rho_{f, q}) = \Sym^{m}(\rho_{g, q}) \otimes \chi$ for some finite order character $\chi$, using the same result of Rajan.

Since the representations $\rho_{f, q}$ and $\rho_{g, q}$ are members of compatible systems, so are their symmetric powers.
Since one member of the compatible system $\Sym^{m}(\rho_{f})$ agrees with one member of the compatible system $\Sym^{m}(\rho_{g}) \otimes \chi$, the whole systems agree; thus we conclude that for the prime $p$ 
\[
\Sym^{m}(\rho_{f, p}) = \Sym^{m}(\rho_{g, p}) \otimes \chi.
\]
We know that the $p$-adic Galois representation of a $p$-ordinary form is upper triangular when restricted to the decomposition group at $p$.
In particular we know that $\rho_{f, p}|G_{\Q_{p}}$ has the form
\[
\begin{bmatrix} \cyclo_{p}^{k-1}\psi_{f} & \ast \\ 0 & \lambda_{f} \end{bmatrix}
\]
for some characters $\psi_{f}, \lambda_{f}$; similarly $\rho_{g, p}|G_{\Q_{p}}$ is of that form with characters $\psi_{g}, \lambda_{g}$.
Since the symmetric powers agree up to twist, we have an equality of sets of characters
\[
\{\cyclo_{p}^{i(k-1)}\psi_{f}^{i} \lambda_{f}^{m-i}: i = 0, \ldots, m\}	= \{\cyclo_{p}^{i(k-1)}\psi_{g}^{i}\lambda_{g}^{m-i}\chi: i = 0, \ldots, m\}.
\]
Note that $\lambda_{f}, \lambda_{g}$ are unramified, and $\psi_{f}, \psi_{g}$ have finite order on inertia with $\psi_{f}\neq\psi_{g}$ on inertia by the choice of $\zeta \neq 1$.
By comparing powers of the cyclotomic character which appear in the above equality of sets of characters, we conclude that $\psi_{f}^{i}\lambda_{f}^{m-i} = \psi_{g}^{i}\lambda_{g}^{m-i}\chi$ for each $i = 0, \ldots, m$.
Rearranging we get that
\[
\frac{\psi_{f}^{i}}{\psi_{g}^{i}} = \frac{\lambda_{g}^{m-i}}{\lambda_{f}^{m-i}} \chi
\]
for each $i = 0, \ldots, m$; in particular when restricted to inertia this yields $\frac{\psi_{f}^{i}}{\psi_{g}^{i}} = \chi$ for each $i$.
Taking $i=0$ we see that $\chi$ is unramified, but taking $i = 1$ shows that $\chi$ must be non-trivial on inertia.
This is a contradiction.

The only assumption that we made outside of the original hypotheses was that neither $f$ nor $g$ has CM, in order to use the fact that Galois representations attached to non-CM forms have large image.
Since we arrived at a contradiction it must be the case that at least one of them has CM, and hence the whole component $\mathbf{I}$ has CM by \cref{prop:cm_component_unique}.
% We note that on inertia, $\psi_{f}$ is a positive power of the cyclotomic character times a finite order character, and $\psi_{g}$ is $\psi_{f}$ times the character associated to $\zeta$.
% By comparing which powers of the cyclotomic character appear, we conclude that in fact $\psi_{f}^{i}\lambda_{f}^{m-i} = \psi_{g}^{i}\lambda_{g}^{m-i}\chi$ for each $i = 0, \ldots, m$.

% Since the Galois representation $\rho_{\mathbf{I}}$ has fixed determinant, we also know that $\psi_{f}\lambda_{f} = \psi_{g}\lambda_{g}$.
% Let $\gamma$ denote the quotient character $\psi_{g}/\psi_{f} = \lambda_{f}/\lambda_{g}$.
% Combining the equality of characters from the determinant with the equality of characters from the symmetric powers, we see that $\gamma^{m-2j} \chi = 1$ for all $j = 0, \ldots, m$.
% This is a contradiction as taking $j = 0$ and $j = 1$ shows that $\gamma^{2} = 1$, but for varying choices of $\zeta$ we have that the quotient $\gamma = \psi_{g}/\psi_{f}$ can have arbitrarily large order.

% The only assumption that we made outside of the original hypotheses was that neither $f$ nor $g$ has CM.
% Since we arrived at a contradiction it must be the case that at least one of them has CM, and hence the whole component $\mathbf{I}$ has CM.
\end{proof}

\begin{remark}
While Hida's proof is written in the case of parallel weight $k \geq 2$ Hilbert modular forms, and we've only sketched the argument in the case of elliptic modular forms, the argument applies equally well to a partially ordinary family and fixed regular weight $(\underline{k}, w)$.

The only difference in the argument comes right at the end when extracting a contradiction from the equality of sets of characters.
In the partially ordinary case it is natural to work with Galois representations having a fixed determinant, rather than determinant varying with the weight as is the common choice for elliptic modular forms.
Since we're working with a fixed weight and varying the Nebentypus character, there's no obstruction to still matching up characters based on the power of the cyclotomic character that appears (i.e. based on their Hodge-Tate weights).
From there an slightly modified argument from the elliptic provides a contradiction, so long as the chosen $\zeta$ has sufficiently large order.
\end{remark}

\begin{remark}\label{rem:hida_theorem_weight_one}
One might ask if \cref{thm:hida_cm_iff_hecke_field} can be applied to a set of classical weight one eigenforms arising from $\mathbf{I}$ and having appropriately bounded Hecke fields.
There are two places where the regular weight assumption is used in the proof.
First, the fact that the Frobenius eigenvalues in the Galois representation attached to a regular weight form are Weil numbers, which is a consequence of the Ramanujan conjecture.
Second, the fact that the $\ell$-adic Galois representation attached to a non-CM eigenform of regular weight has large image for sufficiently large $\ell$.

If we are willing to use the (known!) Ramanujan conjecture for weight one eigenforms then the first use of the regular weight assumption can be taken care of.
The second presents more difficulty in generalizing directly to the weight one case.
However, we expect that given the strong control over Frobenius eigenvalues across the entire component $\mathbf{I}$ provided by the Ramanujan conjecture it should be possible to start with bounded Hecke fields in weight one to establish rigidity of the Frobenius eigenvalues, and then carry out the rest of Hida's argument in regular weight.

We remark once again that the role of \cref{sec:rigidity} and \cref{sec:construction} is to provide a method by which bounds on Hecke fields may be propagated from low weight into regular weight.
% While this may be redundant for elliptic modular eigenforms, it is necessary for the case of Hilbert modular eigenforms of partial weight one, as the Ramanujan conjecture is not yet resolved for such forms so Hida's \cref{thm:hida_cm_iff_hecke_field} cannot use them as an input since it relies crucially on finiteness properties of Weil numbers.
\end{remark}

\subsection{Proof of the main theorem}\label{sec:proof}

We are finally in a position to assemble all the ingredients of the previous two sections in order to prove our main theorem.
We remind the reader that the main goal is to prove that a component $\mathbf{I}$ of $\mathbf{H}^{\ord}$ has CM if it admits infinitely many classical weight one specializations, which we do by propagating information about Hecke fields from low weight into regular weight so that we many apply Hida's result \cref{thm:hida_cm_iff_hecke_field}.

\begin{theorem}\label{thm:classical_pw1_implies_hida}
A reduced irreducible component $\mathbf{I}$ of $\mathbf{H}^{\ord}$ has CM if and only if it admits infinitely many classical weight one specializations.
\end{theorem}
\begin{proof}
If $\mathbf{I}$ has CM, then any specialization in a classical weight is a classical CM form; in particular there will be infinitely many classical weight one eigenforms arising as specializations of a CM component $\mathbf{I}$.
Thus the real task in this proof is to show that if $\mathbf{I}$ admits infinitely many classical weight one specializations, then $\mathbf{I}$ has CM.

We show that the conditions of \cref{thm:hida_cm_iff_hecke_field} apply if we are given such an infinite set of classical weight one specializations, which is now just a matter of assembling the ingredients of \cref{sec:rigidity} and \cref{sec:construction}.
In fact we'll show a stronger statement than necessary to apply \cref{thm:hida_cm_iff_hecke_field}.
For each $\ell \nmid Np$ we produce a constant $C_{\ell}$ such that for almost all classical specializations $f$ of $\mathbf{I}$ we have that
\[
[\Q(\mu_{p^{\infty}}, a_{\ell}(f)): \Q(\mu_{p^{\infty}})] \leq C_{\ell}.
\]

We recall that the main result of \cref{sec:construction} was the construction of a Galois representation $r: G_{\Q} \to \GL_{2m}(\Frac{M})$ for some integral extension $M$ of $\Lambda$.
The key property of $r$ is that for some infinite set $R$ of classical weight one specializations of $\mathbf{I}$ we have for $f \in R$ that
\[
r: G_{\Q} \to \GL_{2m}(\Frac(M)) \overset{P_{f}}{\to} \GL_{2m}(\Qbar_{p})
\]
is well-defined, and equal to the direct sum of the $p$-adic Galois representations attached to the external Galois conjugates $f_{1}, \ldots, f_{m}$ of $f$ over its character field.
In \cref{lem:coefficient_rigidity} we showed that each coefficient $A_{\ell, j}$ of the characteristic polynomial $A_{\ell}(X) = \sum_{j=0}^{2m} A_{\ell, j}X^{j}$ of $r(\Frob_{\ell})$ is controlled by the rigidity results of \cref{sec:rigidity}.
Namely, for each $j = 0, \ldots, 2m$ we have an expression
\[
A_{\ell, j} = \sum_{i=1}^{n_{j}} d_{i, j} (1 + T)^{e_{i, j}}
\]
for some algebraic $d_{i, j}$ and $p$-adic integers $e_{i, j}$.

A final key feature of the representation $r$ is that it ``sees'' almost all classical specializations of $\mathbf{I}$.
If $f$ is a classical eigenform arising as the specialization of $\mathbf{I}$ through a $\Qbar_{p}$ point $P$, denote by $P$ again an extension of this point to $M$.
Then for almost all $P$ we have that the image of $r$ lands in $\GL_{2m}(M_{P})$ (as in part (5) of \cref{thm:components}), so we can push forward $r$ through $P$ to obtain a representation into $\GL_{2m}(\Qbar_{p}$).
By the construction of $r$ as a direct sum of representations attached to components of $\mathbf{H}^{\ord}$, we see that $\rho_{f, p}$ is a direct summand of this specialization of $r$.
In particular this shows that the Frobenius eigenvalues $\alpha_{f}, \beta_{f}$ of $\rho_{f, p}(\Frob_{\ell})$ are roots of $A_{\ell}(P)(X)$.
Since the other direct summands of this specialization of $r$ are the Galois representations attached to other classical forms, we may also conclude that all coefficients of $A_{\ell}(P)(X)$ are algebraic.

We now make use of our exact formula for the $A_{\ell, j}$, as explained in \cref{sec:exponential}.
Define $\pi_{i, j} = (1+p)^{e_{i, j}}$.
Since almost all specializations of $\mathbf{I}$ in weights $k \geq 2$ are classical and are witnessed by $r$, we have that almost all specializations $A_{\ell, j}(P_{k, \zeta})$ are algebraic, as these specializations are polynomial combinations of Frobenius eigenvalues of classical forms.
Thus \cref{prop:pi_algebraic} shows that the $\pi_{i, j}$ are all algebraic.

We conclude that for almost all points $P_{k, \zeta}$ of $\mathbf{I}$ for $k \geq 2$ satisfy that $A_{\ell}(P_{k, \zeta})(X)$ has coefficients in $L_{\ell}(\zeta)$, where
\[
L_{\ell} = \Q(\{d_{i, j}\}, \{\pi_{i, j}\}).
\]
Note that $L_{\ell}$ is a finite extension of $\Q$ since we've adjoined finitely many algebraic quantities to $\Q$.
Therefore for almost all $f$ arising as specializations of $\mathbf{I}$ we have that the eigenvalues of $\rho_{f, p}(\Frob_{\ell})$, and hence also the $T_{\ell}$-eigenvalue $a_{\ell}(f)$, lie in a degree at most $2m$ extension of $L_{\ell}(\zeta)$, since they are roots the degree $2m$ polynomial $A_{\ell}(P)(X)$ which has coefficients in $L_{\ell}(\zeta)$.
Adjoining all $p$-power roots of unity, we see that $a_{\ell}(f)$ has degree at most $2m[L_{\ell}:\Q]$ over $\Q(\mu_{p^{\infty}})$.

Define $C_{\ell}$ to be this constant
\[
C_{\ell} = 2m[L_{\ell}:\Q]
\]
which depends only on $\mathbf{I}$ and $\ell$, and not on our choice of regular weight specialization.
We now have that \cref{thm:hida_cm_iff_hecke_field} applies to $\mathbf{I}$ with any positive density subset of primes $\ell \nmid Np$, these choices of $C_{\ell}$, any fixed choice of $k \geq 2$, and a choice of any infinite subset $S$ of $\mu_{p^{\infty}}$ which avoids a finite set of pairs $(k, \zeta)$ where the representation $r$ has poles.
Thus we conclude that $\mathbf{I}$ has CM!
\end{proof}

\section{Hilbert modular forms of partial weight one}\label{sec:hilbert}

In this section we discuss extensions of our method to the case of partial weight one Hilbert modular forms.
Fix a totally real field $F$ of degree $d = [F:\Q]$ in which our prime $p$ splits completely.
We will discuss Hilbert modular forms for the field $F$.

We think of the weights of a Hilbert modular form as a tuple of integers indexed by the real embeddings of $F$, along with an extra parameter to fix the transformation law, so a Hilbert modular form over $F$ has $d+1$ weights which are usually notated $(k_{1}, \ldots, k_{d}, w)$.
By fixing an isomorphism $\C \to \C_{p}$, we get a bijection between the real and $p$-adic embeddings of $f$, so we can equally well think of the first $d$ weights as being indexed by the $p$-adic places of $f$.

For each of the $p$-adic places $v_{1}, \ldots, v_{d}$ of $F$, there is an operator $U_{v_{i}}$ acting on Hilbert modular forms of level divisible by $v_{i}$.
Suppose that $f$ is an eigenform for the operators $U_{v_{i}}$, of weight $(k_{1}, \ldots, k_{d}, w)$.
Under a suitable choice of normalization for these operators depending on $w$, we have for $a_{v_{i}}$ the eigenvalue of $U_{v_{i}}$ acting on $f$ that
\[
0 \leq \ord_{p}(a_{v_{i}}) \leq k_{i} - 1.
\]
For a $p$-adic place $v$ of $F$ we say that a Hilbert modular eigenform is $v$-ordinary if its $U_{v}$ eigenvalue is a $p$-adic unit; we say that $f$ is $p$-ordinary if it is $v$-ordinary for each $v|p$.
Since the operators $U_{v_{i}}$ commute we can equally define a $p$-ordinary Hilbert modular form to be one whose $U_{p} = \prod_{i=1}^{d} U_{v_{i}}$ eigenvalue is a $p$-adic unit.

Hida has constructed families of $p$-ordinary Hilbert modular forms.
The weight space for these families of forms is now $d+1$-dimensional, and much of the theory that is familiar in the elliptic case is also known in the Hilbert case.
See \cite{hida_hilbert_i} and \cite{hida_hilbert_ii} for the construction of these families.

% where do things break down? Galois conjugation
% direct analog of prop x no longer holds, since Galois conjugation may wreck ordinarity at other primes
\subsection{Galois conjugates of partial weight one forms}\label{sec:hilbert_galois_conj}

We lay out here why the main strategy of this article does not immediately produce results for partial weight one forms.
The issue lies in generalizing \cref{cor:weight_one_conjugates}.

Suppose that we have a Hilbert modular eigenform $f$, of partial weight one.
Explicitly let us say that the first $r$ weights $k_{1}, \ldots, k_{r}$ are equal to $1$, with the others $k_{r+1}, \ldots, k_{d}$ being greater than $1$.
Consider a Galois conjugate $f^{\sigma}$ of $f$ for some $\sigma$ in the absolute Galois group of the character field of $f$ over $F$.
This form $f^{\sigma}$ has the same weights as $f$, and its Hecke eigenvalues are the Galois conjugates by $\sigma$ of those of $f$.

Let us assume now that our starting form $f$ is $p$-ordinary.
By the same argument as applied in \cref{cor:weight_one_conjugates}, we have that $f^{\sigma}$ is still $v_{i}$-ordinary at those $i$ for which $k_{i} = 1$.
However, all we know at the other $p$-adic places is that the $U_{v_{i}}$ eigenvalue has valuation bounded between $0$ and $k_{i} - 1 > 0$.
This is why our method does not produce results immediately for partial weight one forms living in Hida's full $p$-ordinary families: Galois conjugation may ruin ordinarity at non-weight-$1$ places.

The fact that Galois conjugates of elliptic $p$-ordinary eigenforms of weight one remain $p$-ordinary was crucial to our strategy of characterizing families containing many of these by studying their Hecke fields.
In order to apply Hida's characterization of CM families as those having bounded Hecke fields, we need some input to bound our Hecke fields in low weight, so analogs of \cref{cor:weight_one_conjugates} and \cref{lem:weight_one_hecke_field_bound} are necessary for our strategy to function.

\subsection{Partially ordinary families}\label{sec:partially_ordinary}

However, all is not lost in generalizing \cref{cor:weight_one_conjugates} to the partial weight one case.
The argument outlined above does show that if $f$ has its first $r$ weights equal to $1$, and is $v_{i}$-ordinary at those corresponding $p$-adic places, then appropriate Galois conjugates of $f$ will still be $v_{i}$-ordinary for each $i = 1, \ldots, r$.
Thus if we choose to work with families of forms which are $v_{i}$-ordinary for $i = 1, \ldots, r$, we can hope to recover a uniform bound on Hecke fields as in \cref{lem:weight_one_hecke_field_bound} for the forms in this family which are weight one at $v_{1}, \ldots, v_{r}$.
The setback here is that these ``partially ordinary'' families are not at all well-studied!

In \cite{wilson}, a candidate for a partially ordinary Hecke algebra is constructed using the algebraic approach to Hida theory: Wilson works with algebraic automorphic forms in the Betti cohomology of quaternionic Shimura varieties of dimension $0$ or $1$, depending on the parity of $d = [F:\Q]$.
This construction produces an algebra partially ordinary Hecke algebra which is torsion-free over an appropriate Iwasawa algebra, and for which the control theorem is known only up to a finite kernel.
Hida's article \cite{hida_hecke_fields_of_hilbert_modular_analytic_families} also works with partially ordinary families.
Hida sketches some parts of their construction, but references \cite{wilson} for many of their properties.
In this article Hida states without proof the existence of the Galois representations attached to these partially ordinary families.
Partially ordinary families are also discussed briefly in \cite{thorne_dihedral}, again working cohomologically with quaternionic automorphic forms.
We should mention that similar families have been constructed independently by Yamagami in \cite{yamagami} and by Johansson--Newton in \cite{johansson_newton}.
Both Yamagami and Johansson--Newton work with more general $v$-finite-slope rather than $v$-ordinary families, though one can think of the $\ord_{p}(U_{v}) = 0$ locus in the $v$-finite-slope eigenvariety as being the $v$-ordinary families that we study.

Worse than the fact that the literature on partially ordinary families is not as mature as that of fully ordinary families is the fact that all existing constructions of partially ordinary families are algebraic, rather than geometric.
The algebraic description has the downside that it does not give any information about the inclusion of forms of partial weight one into these families.
Since it is crucial for us to know that every suitably $v$-ordinary partial weight one form is included in such a family, we believe that a geometric construction of these families is needed in order to apply them in our situation.

The strategy of this article is adapted to deal with $1$-dimensional families.
We state below what we expect our strategy to prove given a suitable construction of $1$-dimensional partially ordinary families, and in the next section we outline how one might combine the result for $1$-dimensional families with a putative classicality result for partial weight one forms into a statement about families of any dimension.

% what do we prove
Fix weights $k_{2}, \ldots, k_{d}, w$, all of which are odd and at least $3$.
We will consider forms of weight $(k, k_{2}, \ldots, k_{d}, w)$, and we want the fixed weights to be odd so that the partial weight one forms we work with are paritious (all of the weights have the same parity, which guarantees algebraicity of Hecke eigenvalues and $\GL_{2}$ rather than $\PGL_{2}$ valued Galois representations).
The first weight will be the one that varies across the family while the others remain constant.
We also fix a tame level $\mathfrak{N}$, which may be divisible by any of $v_{2}, \ldots,  v_{d}$ but is not divisible by $v_{1}$.
We suppose that we have constructed a partially ordinary Hecke algebra $\mathbf{H}$ satisfying the following properties:
\begin{itemize}
\item For each prime $\mathfrak{l} \nmid \mathfrak{N}v_{1}$ of F, there is an element $T_{\mathfrak{l}} \in \mathbf{H}$.
\item For each prime $\mathfrak{l} | \mathfrak{N}v_{1}$ of F, there is an element $U_{\mathfrak{l}} \in \mathbf{H}$.
\item $\mathbf{H}$ is a finite free $\Lambda = \Z_{p}\llbracket T \rrbracket$ module.
\item $\mathbf{H}$ is the ``universal Hecke algebra'' parametrizing $v_{1}$-ordinary Hilbert modular eigenforms of weight $(k, k_{2}, \ldots, k_{d}, w)$ and levels $\mathfrak{N}v_{1}^{r}$.
More precisely, an analog of \cref{thm:hida_theory_main_statement} holds for $\mathbf{H}$.
\item Associated to any component $\mathbf{I}$ of $\mathbf{H}$ is a Galois representation $G_{F} \to \GL_{2}(\Frac(\mathbf{I}))$ interpolating the Galois representations of the eigenforms which $\mathbf{I}$ specializes to.
An analog of \cref{thm:hida_theory_galois_representation} holds for these representations.
\item Any $v_{1}$-ordinary Hilbert modular eigenform of weight $(1, k_{2}, \ldots, k_{d}, w)$ arises as a specialization of $\mathbf{H}$.
An analog of \cref{prop:hida_theory_weight_one} is what is required.
\end{itemize}

Given such a family $\mathbf{H}$ we expect that the strategy of this article immediately adapts to prove the following result: a component $\mathbf{I}$ of $\mathbf{H}$ contains infinitely many classical eigenforms of partial weight one if and only if that component has complex multiplication.

\subsection{Returning to full families}\label{sec:full_families}

% restriction + classicality result
We turn now to the question of when Hida's full $p$-ordinary families contain many classical forms of low weight.
We note that the original technique of Ghate--Vatsal does adapt to the case of families containing a Zariski dense set of classical parallel weight one forms, the details of which are worked out in \cite{balasubramanyam_ghate_vatsal}.
We therefore focus on explaining how to recover a result characterizing full families which contain a Zariski dense set of classical partial weight one forms under the assumption of the result from the previous section for $1$-dimensional families along with a suitable classicality theorem.

By analogy with elliptic forms of weight one, we expect that forms which are weight one at $v$ and $v$-ordinary in fact have Galois representations which are split, rather than just upper triangular, on a decomposition group at $v$.
It may be possible to extract this fact from the literature, though this has not been clearly stated anywhere to the author's knowledge.
Moreover we expect that this splitting of the local Galois representation characterizes classical forms of partial weight one among $p$-adic forms.
To be more precise, we expect that ``classicality'' results of the following type hold: suppose that $f$ is a $p$-adic Hilbert modular eigenform for which
\begin{itemize}
\item the weight $(k_{1}, \ldots, k_{d}, w)$ of $f$ is arithmetic (the $k_{i}$ are paritious positive integers),
\item for those $i$ with $k_{i} > 1$, $f$ has finite $U_{v_{i}}$ slope in the range $[0, k_{i} - 1)$ (analog of the classical Coleman classicality condition \cite{coleman}),
\item and for those $i$ with $k_{i} = 1$, the Galois representation attached to $f$ is split when restricted to a decomposition group at $v_{i}$;
\end{itemize}
then (perhaps assuming some technical conditions on $f$ and its Galois representation) $f$ is classical.

The literature on classicality theorems for (Hilbert) modular forms is well-developed, though as yet no results have appeared for forms of partial weight one.
The basic techniques used are analytic continuation and gluing of overconvergent eigenforms, as pioneered by Buzzard--Taylor in \cite{buzzard_taylor}, Buzzard in \cite{buzzard}, and Kassaei in \cite{kassaei}.
Much work has been done on extending these techniques to the Hilbert setting; the simplest case is when $p$ splits in the extension $F/\Q$ where one can essentially apply the arguments of the elliptic case ``prime by prime'' to establish analytic continuation over a large enough region of the relevant Hilbert modular variety.
We expect that once the ecosystem for dealing with partial weight one forms is sufficiently developed this prime by prime approach and the standard techniques for establishing classicality results should prove such a result for partial weight one forms.

Suppose that we have such a classicality result for partial weight one forms and a construction of partially ordinary families as outlined in the previous section.
If we have a component of a full $p$-ordinary family which has a Zariski dense set of classical specializations which are weight one at one of the primes $v|p$, then by the Zariski density of these specializations the Galois representation attached to the component must split on a decomposition group at $v$.
By the classicality result, we then get that \emph{any} point of the family which is weight one $v$ and regular weight elsewhere must be classical.
In particular if we restrict to a $1$-dimensional slice of our component where only the weight at $v$ is allowed to vary, our result for $1$-dimensional families shows that that slice of the full family has complex multiplication, as that $1$-dimensional slice is nothing other than a component of a $v$-ordinary family.
Since there are only finitely many possible imaginary quadratic extensions of $F$ which a family with a given tame level could have complex multiplication by, we see that infinitely many of our $1$-dimensional slices must have CM by the same imaginary quadratic $E/F$.
But then we have that our full family must overlap with a CM family at a Zariski dense set of these $1$-dimensional slices, meaning our full $p$-ordinary family must in fact have CM.

We note a pleasing consequence of this result at the level of Galois representations: the argument above shows that if the Galois representation attached to a component of a full $p$-ordinary family splits on a decomposition group at one place $v|p$, then that component must have CM.
But CM families which are $p$-ordinary have Galois representations which are split at each place above $p$!
So if the Galois representation of a component splits at one place above $p$ it must split at all of them.

\printbibliography

@article {cassels,
    AUTHOR = {Cassels, J. W. S.},
     TITLE = {On a conjecture of {R}. {M}. {R}obinson about sums of roots of
              unity},
   JOURNAL = {J. Reine Angew. Math.},
  FJOURNAL = {Journal f\"{u}r die Reine und Angewandte Mathematik},
    VOLUME = {238},
      YEAR = {1969},
     PAGES = {112--131},
      ISSN = {0075-4102},
   MRCLASS = {10.66},
  MRNUMBER = {0246852},
MRREVIEWER = {W. G. H. Schaal},
       DOI = {10.1515/crll.1969.238.112},
       URL = {https://doi.org/10.1515/crll.1969.238.112},
}

@article {loxton,
    AUTHOR = {Loxton, J. H.},
     TITLE = {On the maximum modulus of cyclotomic integers},
   JOURNAL = {Acta Arith.},
  FJOURNAL = {Polska Akademia Nauk. Instytut Matematyczny. Acta Arithmetica},
    VOLUME = {22},
      YEAR = {1972},
     PAGES = {69--85},
      ISSN = {0065-1036},
   MRCLASS = {12A10},
  MRNUMBER = {0309896},
MRREVIEWER = {J. B. Kelly},
       DOI = {10.4064/aa-22-1-69-85},
       URL = {https://doi.org/10.4064/aa-22-1-69-85},
}

@article {ghate_vatsal,
    AUTHOR = {Ghate, Eknath and Vatsal, Vinayak},
     TITLE = {On the local behaviour of ordinary {$\Lambda$}-adic
              representations},
   JOURNAL = {Ann. Inst. Fourier (Grenoble)},
  FJOURNAL = {Universit\'{e} de Grenoble. Annales de l'Institut Fourier},
    VOLUME = {54},
      YEAR = {2004},
    NUMBER = {7},
     PAGES = {2143--2162 (2005)},
      ISSN = {0373-0956},
   MRCLASS = {11F80 (11F33 11R23)},
  MRNUMBER = {2139691},
MRREVIEWER = {Laurent N. Berger},
       URL = {http://aif.cedram.org/item?id=AIF_2004__54_7_2143_0},
}

@article {balasubramanyam_ghate_vatsal,
    AUTHOR = {Balasubramanyam, Baskar and Ghate, Eknath and Vatsal, Vinayak},
     TITLE = {On local {G}alois representations associated to ordinary
              {H}ilbert modular forms},
   JOURNAL = {Manuscripta Math.},
  FJOURNAL = {Manuscripta Mathematica},
    VOLUME = {142},
      YEAR = {2013},
    NUMBER = {3-4},
     PAGES = {513--524},
      ISSN = {0025-2611},
   MRCLASS = {11F80 (11F33 11F41 11F85 11R23)},
  MRNUMBER = {3117174},
MRREVIEWER = {Anton Deitmar},
       DOI = {10.1007/s00229-013-0614-1},
       URL = {https://doi.org/10.1007/s00229-013-0614-1},
}

@article {hida_hecke_fields_of_analytic_families_of_modular_forms,
    AUTHOR = {Hida, Haruzo},
     TITLE = {Hecke fields of analytic families of modular forms},
   JOURNAL = {J. Amer. Math. Soc.},
  FJOURNAL = {Journal of the American Mathematical Society},
    VOLUME = {24},
      YEAR = {2011},
    NUMBER = {1},
     PAGES = {51--80},
      ISSN = {0894-0347},
   MRCLASS = {11F25 (11E16 11F11 11F33 11F85)},
  MRNUMBER = {2726599},
MRREVIEWER = {Fabrizio Andreatta},
       DOI = {10.1090/S0894-0347-2010-00680-7},
       URL = {https://doi.org/10.1090/S0894-0347-2010-00680-7},
}

@incollection {hida_hecke_fields_of_hilbert_modular_analytic_families,
    AUTHOR = {Hida, Haruzo},
     TITLE = {Hecke fields of {H}ilbert modular analytic families},
 BOOKTITLE = {Automorphic forms and related geometry: assessing the legacy
              of {I}. {I}. {P}iatetski-{S}hapiro},
    SERIES = {Contemp. Math.},
    VOLUME = {614},
     PAGES = {97--137},
 PUBLISHER = {Amer. Math. Soc., Providence, RI},
      YEAR = {2014},
   MRCLASS = {11F11 (11E16 11F25 11F27 11F30 11F33 11F80)},
  MRNUMBER = {3220926},
MRREVIEWER = {J. Larry Lehman},
       DOI = {10.1090/conm/614/12251},
       URL = {https://doi.org/10.1090/conm/614/12251},
}

@article {hida_transcendence,
    AUTHOR = {Hida, Haruzo},
     TITLE = {Transcendence of {H}ecke operators in the big {H}ecke algebra},
   JOURNAL = {Duke Math. J.},
  FJOURNAL = {Duke Mathematical Journal},
    VOLUME = {163},
      YEAR = {2014},
    NUMBER = {9},
     PAGES = {1655--1681},
      ISSN = {0012-7094},
   MRCLASS = {11E25 (11F27 11F41 11F80)},
  MRNUMBER = {3217764},
MRREVIEWER = {Fausto Jarqu\'{\i}n Z\'{a}rate},
       DOI = {10.1215/00127094-2690478},
       URL = {https://doi.org/10.1215/00127094-2690478},
}

@incollection {hida_growth_of_hecke_fields,
    AUTHOR = {Hida, Haruzo},
     TITLE = {Growth of {H}ecke fields along a {$p$}-adic analytic family of
              modular forms},
 BOOKTITLE = {Families of automorphic forms and the trace formula},
    SERIES = {Simons Symp.},
     PAGES = {129--173},
 PUBLISHER = {Springer, [Cham]},
      YEAR = {2016},
   MRCLASS = {11F80 (11E25 11F27 11F41)},
  MRNUMBER = {3675166},
MRREVIEWER = {Cherng-tiao Perng},
       DOI = {10.1142/9761},
       URL = {https://doi.org/10.1142/9761},
}

@article {serban_manin-mumford,
    AUTHOR = {Serban, Vlad},
     TITLE = {An infinitesimal {$p$}-adic multiplicative {M}anin-{M}umford
              conjecture},
   JOURNAL = {J. Th\'{e}or. Nombres Bordeaux},
  FJOURNAL = {Journal de Th\'{e}orie des Nombres de Bordeaux},
    VOLUME = {30},
      YEAR = {2018},
    NUMBER = {2},
     PAGES = {393--408},
      ISSN = {1246-7405},
   MRCLASS = {11S31 (13F25 13H05 14L05)},
  MRNUMBER = {3891318},
MRREVIEWER = {Kevin P. Keating},
       URL = {http://jtnb.cedram.org/item?id=JTNB_2018__30_2_393_0},
}

@inproceedings{serban_bianchi,
  title={A finiteness result for \$p\$-adic families of Bianchi modular forms},
  author={Vlad Serban},
  year={2019}
}

@article{rajan,
    author = {Rajan, C. S.},
    title = "{On strong multiplicity one for l-adic representations}",
    journal = {International Mathematics Research Notices},
    volume = {1998},
    number = {3},
    pages = {161-172},
    year = {1998},
    month = {01},
    issn = {1073-7928},
    doi = {10.1155/S1073792898000142},
    url = {https://doi.org/10.1155/S1073792898000142},
    eprint = {https://academic.oup.com/imrn/article-pdf/1998/3/161/2140278/1998-3-161.pdf},
}

@book {lang_cyclotomic_fields,
    AUTHOR = {Lang, Serge},
     TITLE = {Cyclotomic fields {I} and {II}},
    SERIES = {Graduate Texts in Mathematics},
    VOLUME = {121},
   EDITION = {second},
      NOTE = {With an appendix by Karl Rubin},
 PUBLISHER = {Springer-Verlag, New York},
      YEAR = {1990},
     PAGES = {xviii+433},
      ISBN = {0-387-96671-4},
   MRCLASS = {11-02 (11R18 11R20 11S40 11T22)},
  MRNUMBER = {1029028},
MRREVIEWER = {T. Mets\"{a}nkyl\"{a}},
       DOI = {10.1007/978-1-4612-0987-4},
       URL = {https://doi.org/10.1007/978-1-4612-0987-4},
}

@book {koblitz_p-adic,
    AUTHOR = {Koblitz, Neal},
     TITLE = {{$p$}-adic numbers, {$p$}-adic analysis, and zeta-functions},
    SERIES = {Graduate Texts in Mathematics},
    VOLUME = {58},
   EDITION = {Second},
 PUBLISHER = {Springer-Verlag, New York},
      YEAR = {1984},
     PAGES = {xii+150},
      ISBN = {0-387-96017-1},
   MRCLASS = {11Q25 (11S80)},
  MRNUMBER = {754003},
       DOI = {10.1007/978-1-4612-1112-9},
       URL = {https://doi.org/10.1007/978-1-4612-1112-9},
}

@book {wilson,
    AUTHOR = {Wilson, Brad Lee},
     TITLE = {p-adic {H}ecke algebras and {L}-functions},
      NOTE = {Thesis (Ph.D.)--University of California, Los Angeles},
 PUBLISHER = {ProQuest LLC, Ann Arbor, MI},
      YEAR = {1993},
     PAGES = {62},
   MRCLASS = {Thesis},
  MRNUMBER = {2688967},
       URL =
              {http://gateway.proquest.com/openurl?url_ver=Z39.88-2004&rft_val_fmt=info:ofi/fmt:kev:mtx:dissertation&res_dat=xri:pqdiss&rft_dat=xri:pqdiss:9317486},
}

@article {dimitrov,
    AUTHOR = {Dimitrov, Mladen},
     TITLE = {Galois representations modulo {$p$} and cohomology of
              {H}ilbert modular varieties},
   JOURNAL = {Ann. Sci. \'{E}cole Norm. Sup. (4)},
  FJOURNAL = {Annales Scientifiques de l'\'{E}cole Normale Sup\'{e}rieure. Quatri\`eme
              S\'{e}rie},
    VOLUME = {38},
      YEAR = {2005},
    NUMBER = {4},
     PAGES = {505--551},
      ISSN = {0012-9593},
   MRCLASS = {11F80 (11F41 14G35)},
  MRNUMBER = {2172950},
MRREVIEWER = {Fabrizio Andreatta},
       DOI = {10.1016/j.ansens.2005.03.005},
       URL = {https://doi.org/10.1016/j.ansens.2005.03.005},
}

@article {buzzard,
    AUTHOR = {Buzzard, Kevin},
     TITLE = {Analytic continuation of overconvergent eigenforms},
   JOURNAL = {J. Amer. Math. Soc.},
  FJOURNAL = {Journal of the American Mathematical Society},
    VOLUME = {16},
      YEAR = {2003},
    NUMBER = {1},
     PAGES = {29--55},
      ISSN = {0894-0347},
   MRCLASS = {11F33 (11F11 11F80 11G18)},
  MRNUMBER = {1937198},
MRREVIEWER = {Gebhard B\"{o}ckle},
       DOI = {10.1090/S0894-0347-02-00405-8},
       URL = {https://doi.org/10.1090/S0894-0347-02-00405-8},
}

@article {buzzard_taylor,
    AUTHOR = {Buzzard, Kevin and Taylor, Richard},
     TITLE = {Companion forms and weight one forms},
   JOURNAL = {Ann. of Math. (2)},
  FJOURNAL = {Annals of Mathematics. Second Series},
    VOLUME = {149},
      YEAR = {1999},
    NUMBER = {3},
     PAGES = {905--919},
      ISSN = {0003-486X},
   MRCLASS = {11F33 (11F11 11F80)},
  MRNUMBER = {1709306},
MRREVIEWER = {Gebhard B\"{o}ckle},
       DOI = {10.2307/121076},
       URL = {https://doi.org/10.2307/121076},
}

@article {kassaei,
    AUTHOR = {Kassaei, Payman L.},
     TITLE = {A gluing lemma and overconvergent modular forms},
   JOURNAL = {Duke Math. J.},
  FJOURNAL = {Duke Mathematical Journal},
    VOLUME = {132},
      YEAR = {2006},
    NUMBER = {3},
     PAGES = {509--529},
      ISSN = {0012-7094},
   MRCLASS = {11F33 (11G18)},
  MRNUMBER = {2219265},
MRREVIEWER = {Thomas A. Weston},
       DOI = {10.1215/S0012-7094-06-13234-9},
       URL = {https://doi.org/10.1215/S0012-7094-06-13234-9},
}

@article {wiles_ordinary,
    AUTHOR = {Wiles, A.},
     TITLE = {On ordinary {$\lambda$}-adic representations associated to
              modular forms},
   JOURNAL = {Invent. Math.},
  FJOURNAL = {Inventiones Mathematicae},
    VOLUME = {94},
      YEAR = {1988},
    NUMBER = {3},
     PAGES = {529--573},
      ISSN = {0020-9910},
   MRCLASS = {11F41 (11F80 11R23 11R80)},
  MRNUMBER = {969243},
MRREVIEWER = {Sheldon Kamienny},
       DOI = {10.1007/BF01394275},
       URL = {https://doi.org/10.1007/BF01394275},
}

@article {deligne_serre,
    AUTHOR = {Deligne, Pierre and Serre, Jean-Pierre},
     TITLE = {Formes modulaires de poids {$1$}},
   JOURNAL = {Ann. Sci. \'{E}cole Norm. Sup. (4)},
  FJOURNAL = {Annales Scientifiques de l'\'{E}cole Normale Sup\'{e}rieure. Quatri\`eme
              S\'{e}rie},
    VOLUME = {7},
      YEAR = {1974},
     PAGES = {507--530 (1975)},
      ISSN = {0012-9593},
   MRCLASS = {10D15 (12A65)},
  MRNUMBER = {379379},
MRREVIEWER = {Stephen Gelbart},
       URL = {http://www.numdam.org/item?id=ASENS_1974_4_7_4_507_0},
}

@article {hida_iwasawa_modules,
    AUTHOR = {Hida, Haruzo},
     TITLE = {Iwasawa modules attached to congruences of cusp forms},
   JOURNAL = {Ann. Sci. \'{E}cole Norm. Sup. (4)},
  FJOURNAL = {Annales Scientifiques de l'\'{E}cole Normale Sup\'{e}rieure. Quatri\`eme
              S\'{e}rie},
    VOLUME = {19},
      YEAR = {1986},
    NUMBER = {2},
     PAGES = {231--273},
      ISSN = {0012-9593},
   MRCLASS = {11F33 (11G18 11R23)},
  MRNUMBER = {868300},
MRREVIEWER = {Yasutaka Ihara},
       URL = {http://www.numdam.org/item?id=ASENS_1986_4_19_2_231_0},
}

@article {hida_galois_representations,
    AUTHOR = {Hida, Haruzo},
     TITLE = {Galois representations into {${\rm GL}_2({\bf Z}_p[[X]])$}
              attached to ordinary cusp forms},
   JOURNAL = {Invent. Math.},
  FJOURNAL = {Inventiones Mathematicae},
    VOLUME = {85},
      YEAR = {1986},
    NUMBER = {3},
     PAGES = {545--613},
      ISSN = {0020-9910},
   MRCLASS = {11F11 (11F33 11F85 11R23)},
  MRNUMBER = {848685},
MRREVIEWER = {Ernst-Wilhelm Zink},
       DOI = {10.1007/BF01390329},
       URL = {https://doi.org/10.1007/BF01390329},
}

@article {deligne_weil_i,
    AUTHOR = {Deligne, Pierre},
     TITLE = {La conjecture de {W}eil. {I}},
   JOURNAL = {Inst. Hautes \'{E}tudes Sci. Publ. Math.},
  FJOURNAL = {Institut des Hautes \'{E}tudes Scientifiques. Publications
              Math\'{e}matiques},
    NUMBER = {43},
      YEAR = {1974},
     PAGES = {273--307},
      ISSN = {0073-8301},
   MRCLASS = {14G13},
  MRNUMBER = {340258},
MRREVIEWER = {Nicholas M. Katz},
       URL = {http://www.numdam.org/item?id=PMIHES_1974__43__273_0},
}

@article {hida_hilbert_i,
    AUTHOR = {Hida, Haruzo},
     TITLE = {On {$p$}-adic {H}ecke algebras for {${\rm GL}_2$} over totally
              real fields},
   JOURNAL = {Ann. of Math. (2)},
  FJOURNAL = {Annals of Mathematics. Second Series},
    VOLUME = {128},
      YEAR = {1988},
    NUMBER = {2},
     PAGES = {295--384},
      ISSN = {0003-486X},
   MRCLASS = {11F41 (11F85)},
  MRNUMBER = {960949},
MRREVIEWER = {Sheldon Kamienny},
       DOI = {10.2307/1971444},
       URL = {https://doi.org/10.2307/1971444},
}

@incollection {hida_hilbert_ii,
    AUTHOR = {Hida, Haruzo},
     TITLE = {On nearly ordinary {H}ecke algebras for {${\rm GL}(2)$} over
              totally real fields},
 BOOKTITLE = {Algebraic number theory},
    SERIES = {Adv. Stud. Pure Math.},
    VOLUME = {17},
     PAGES = {139--169},
 PUBLISHER = {Academic Press, Boston, MA},
      YEAR = {1989},
   MRCLASS = {11F41 (11F85)},
  MRNUMBER = {1097614},
MRREVIEWER = {Noburo Ishii},
       DOI = {10.2969/aspm/01710139},
       URL = {https://doi.org/10.2969/aspm/01710139},
}

@article {johansson_newton,
    AUTHOR = {Johansson, Christian and Newton, James},
     TITLE = {Parallel weight 2 points on {H}ilbert modular eigenvarieties
              and the parity conjecture},
   JOURNAL = {Forum Math. Sigma},
  FJOURNAL = {Forum of Mathematics. Sigma},
    VOLUME = {7},
      YEAR = {2019},
     PAGES = {Paper No. e27, 36},
   MRCLASS = {11F33 (11G40)},
  MRNUMBER = {4010559},
MRREVIEWER = {Chan-Ho Kim},
       DOI = {10.1017/fms.2019.23},
       URL = {https://doi.org/10.1017/fms.2019.23},
}

@article {yamagami,
    AUTHOR = {Yamagami, Atsushi},
     TITLE = {On {$p$}-adic families of {H}ilbert cusp forms of finite
              slope},
   JOURNAL = {J. Number Theory},
  FJOURNAL = {Journal of Number Theory},
    VOLUME = {123},
      YEAR = {2007},
    NUMBER = {2},
     PAGES = {363--387},
      ISSN = {0022-314X},
   MRCLASS = {11F85 (11F41)},
  MRNUMBER = {2301220},
MRREVIEWER = {Jacques Tilouine},
       DOI = {10.1016/j.jnt.2006.07.013},
       URL = {https://doi.org/10.1016/j.jnt.2006.07.013},
}

@inproceedings {ribet_nebentypus,
    AUTHOR = {Ribet, Kenneth A.},
     TITLE = {Galois representations attached to eigenforms with
              {N}ebentypus},
 BOOKTITLE = {Modular functions of one variable, {V} ({P}roc. {S}econd
              {I}nternat. {C}onf., {U}niv. {B}onn, {B}onn, 1976)},
     PAGES = {17--51. Lecture Notes in Math., Vol. 601},
      YEAR = {1977},
   MRCLASS = {10D15},
  MRNUMBER = {0453647},
MRREVIEWER = {Hiroshi Saito},
}

@article {coleman,
    AUTHOR = {Coleman, Robert F.},
     TITLE = {Classical and overconvergent modular forms},
   JOURNAL = {Invent. Math.},
  FJOURNAL = {Inventiones Mathematicae},
    VOLUME = {124},
      YEAR = {1996},
    NUMBER = {1-3},
     PAGES = {215--241},
      ISSN = {0020-9910},
   MRCLASS = {11F85 (14G20)},
  MRNUMBER = {1369416},
MRREVIEWER = {Fernando Q. Gouv\^{e}a},
       DOI = {10.1007/s002220050051},
       URL = {https://doi-org.proxy3.library.mcgill.ca/10.1007/s002220050051},
}

@book {diamond_shurman,
    AUTHOR = {Diamond, Fred and Shurman, Jerry},
     TITLE = {A first course in modular forms},
    SERIES = {Graduate Texts in Mathematics},
    VOLUME = {228},
 PUBLISHER = {Springer-Verlag, New York},
      YEAR = {2005},
     PAGES = {xvi+436},
      ISBN = {0-387-23229-X},
   MRCLASS = {11Fxx},
  MRNUMBER = {2112196},
MRREVIEWER = {Henri Darmon},
}

@article {thorne_dihedral,
    AUTHOR = {Thorne, Jack A.},
     TITLE = {Automorphy of some residually dihedral {G}alois
              representations},
   JOURNAL = {Math. Ann.},
  FJOURNAL = {Mathematische Annalen},
    VOLUME = {364},
      YEAR = {2016},
    NUMBER = {1-2},
     PAGES = {589--648},
      ISSN = {0025-5831},
   MRCLASS = {11F80 (11F41)},
  MRNUMBER = {3451399},
MRREVIEWER = {N\'{u}ria Vila},
       DOI = {10.1007/s00208-015-1214-z},
       URL = {https://doi-org.proxy3.library.mcgill.ca/10.1007/s00208-015-1214-z},
}

\end{document}